\documentclass[preprint,runningheads]{svjour3}

\let\vec\relax
\DeclareMathAccent{\vec}{\mathord}{letters}{"7E}

\usepackage{soul}
\usepackage{amsmath}
\usepackage{subcaption}
\usepackage{amsfonts}
\usepackage{amssymb}
\usepackage{mathtools}
\usepackage{mathrsfs,bbm}
\usepackage{xcolor}
\usepackage{lmodern}
\usepackage{hyperref}
\usepackage{longtable}

\usepackage{graphicx}
\usepackage{epstopdf}

\usepackage{bm}

\usepackage{algorithm}
\usepackage{algpseudocode}

\usepackage{accents}
\usepackage{cancel} 

\journalname{}
\date{ \phantom{b} \vspace{45mm}\phantom{e}}
\def\makeheadbox{\hspace{-4mm}Version of 5 May 2025}

\newcommand*{\dt}[1]{ \accentset{\mbox{\large\bfseries .}}{#1}}

\newtheorem{assumptions}{Assumptions}

\def\R{{\mathbb R}}

\def\torus{{\mathbb T}}
\def\eps{\varepsilon}

\def\wh{\widehat}
\def\wt{\widetilde}
\def\wt{\overline}

\def\h{{1/2}}
\usepackage{xfrac}

\newcommand\bfI{{\mathbf I}}
\newcommand\bfA{{\mathbf A}}

\newcommand\bfC{{\mathbf C}}
\newcommand\bfD{{\mathbf D}}

\newcommand\bfF{{\mathbf F}}
\newcommand\bfG{{\mathbf G}}
\newcommand\bfH{{\mathbf H}}
\newcommand\bfK{{\mathbf K}}
\newcommand\bfL{{\mathbf L}}
\newcommand\bfM{{\mathbf M}}
\newcommand\bfN{{\mathbf N}}
\newcommand\bfP{{\mathbf P}}

\newcommand\bfR{{\mathbf R}}
\newcommand\bfS{{\mathbf S}}
\newcommand\bfU{{\mathbf U}}
\newcommand\bfV{{\mathbf V}}

\newcommand\bfX{{\mathbf X}}
\newcommand\bfY{{\mathbf Y}}
\newcommand\bfZ{{\mathbf Z}}

\def\eps{\varepsilon}

\def\phi{\varphi}

\def\diag{\mbox{diag}}

\usepackage{enumitem}

\title{High-order robust basis-update \& Galerkin integrators for dynamical low-rank approximation}
\titlerunning{High-order robust BUG integrators for DLRA}
\date{\today}
\author{Cory Hauck, Jonas Kusch, Steffen Schotth\"ofer}

\institute{
  Cory Hauck \at Computer Science and Mathematics Division, Oak Ridge National Laboratory, Oak Ridge, TN 37980, USA
  \email{hauckcd@ornl.gov}
  \and
  J.Kusch \at Scientific Computing, Norwegian University of Life Sciences, Drøbakveien 31, 1433 Ås, Norway. \email{jonas.kusch1@gmail.com}        \\
  \and
  Steffen Schotth\"ofer \at Seamless Labs Inc., Mountain View, CA, USA
  \email{steffen@seamlesslabs.ai}
}

\begin{document}
	\maketitle
	\begin{abstract}  
    Dynamical low-rank approximation has become a widely used numerical method in diverse disciplines. Its main idea is to represent the matrix or tensor-valued solution to a time-dependent differential equation as a low-rank factorization. The evolution of the factorization leads to highly stiff dynamics and requires the derivation of novel time integration methods that are not prone to this stiffness. A promising family of integrators are basis-update \& Galerkin (BUG) integrators as they enable implicit time integration and structure--preservation. However, current BUG integrators are limited to second--order accuracy while general-order BUG integrators are designed as projections of explicit time integration methods, thus severely limiting their use.
    In this work, we propose general-order BUG integrators that do not rely on an explicit time integration scheme while requiring a smaller number of basis functions to achieve high-order accuracy. We prove a general order error bound for the proposed augmented and parallel BUG integrators and demonstrate their behaviour for a series of stiff and non-stiff numerical benchmarks in which they significantly outperform previous BUG versions.

\end{abstract}

\section{Introduction}
 Dynamical low-rank approximation (DLRA)~\cite{KochLubich07} has gained significant attention recently due to its potential to substantially reduce computational costs and memory requirements when solving high-dimensional differential equations. While DLRA is particularly efficient for high-dimensional tensor differential equations \cite{haegeman2016unifying,ceruti2023rank,sulz2024numerical,ceruti2024parallel}, it has also been successfully applied to matrix differential equations, where it has found applications across a wide range of disciplines including kinetic problems~\cite{einkemmer2024review}, chemical kinetics \cite{Jahnke2008,Prugger2023,Einkemmer2023a}, wave propagation \cite{Hochbruck2023,zhao2023low}, uncertainty quantification \cite{SaL09,babaee2017robust,FeL18,MuN18,MuNV20,patil2020real,kusch2021DLRUQ,kazashi2021existence,DoPNH23,ali2024dynamicallyAppl}, and machine learning \cite{schotthofer2022low,zangrando2023rank,schmidt2023rank,schotthofer2024federated,schotthoefer2024geolorageometricintegrationparameter}. 

Given a matrix differential equation of the form
\begin{align}\label{eq:origProb}
    \dot{\bfA}(t) = \bfF(t,\bfA(t))\,,
\end{align}
where $\bfA(t)\in\mathbb{R}^{m\times n}$, DLRA provides an approximate low-rank solution of the form $\bfX(t) =\bfU(t)\bfS(t)\bfV(t)^{\top}$.  Here $\bfS(t)\in\mathbb{R}^{r\times r}$, where $r\ll m,n$ is the approximation rank, and $\bfV(t)\in\mathbb{R}^{n\times r}$ and $\bfU(t)\in\mathbb{R}^{m\times r}$ are basis matrices with orthonormal columns. At its core, DLRA reduces computational complexity by restricting the evolution of $\bfX(t)$ to the manifold of low-rank matrices $\mathcal{M}$. This is achieved by defining the projector $\bfP(\bfX)$ onto the tangent space of $\mathcal{M}$ at $\bfX=\bfU\bfS\bfV^{\top}$:
\begin{align*}
    \bfP(\bfX)\bfZ = \bfU\bfU^{\top}\bfZ (\bfI - \bfV\bfV^{\top}) + \bfZ\bfV\bfV^{\top}\,.
\end{align*}
The low-rank approximation $\bfX(t)$ then solves the differential equation
\begin{align}\label{eq:proj_flow}
    \dot{\bfX}(t) = \bfP(\bfX(t))\bfF(t,\bfX(t))\,.
\end{align}
The key difficulty in solving this differential equation lies in the stiffness introduced by $\bfP$, which results from the high curvature of  $\mathcal{M}$. The stiffness introduced by this curvature is particularly severe when the smallest singular value of the factorization approaches zero, necessitating prohibitively small time steps for standard numerical schemes to maintain stability \cite{KochLubich07}.

To address this challenge, several robust low-rank integrators have been proposed. The projector--splitting (PS) integrator~\cite{LubichOseledets,KieriLubichWalach} and the basis-update \& Galerkin (BUG) integrators~\cite{CeL22,CeKL22,CeKL23,ceruti2024robust} are among the most widely used methods, along with projection-based approaches~\cite{KiV19,charous2023dynamically,SeCK23}. While the PS integrator has demonstrated reliable accuracy in many applications, its backward-in-time evolution substep can introduce instabilities, particularly in problems with parabolic character or strong scattering effects~\cite{kusch2021stability,einkemmer2024asymptotic}. Additionally, its analytically proven error bounds are limited to first order~\cite{KieriLubichWalach}. Projection-based approaches are commonly restricted to explicit time integration methods while exhibiting large error constants and lacking structure--preservation. The BUG integrators offer an alternative approach by updating the basis matrices in parallel to define a linear space in which local coefficients can be found, followed by a sequential update of this coefficient matrix. A key advantage of these methods is that the solution evolves only forward in time, while enabling straightforward extensions to structure preservation \cite{einkemmer2023conservation} and allowing for implicit substeps to update expansion coefficients to stabilize computations. Perhaps the most widely used robust integrator is the augmented BUG integrator \cite{ceruti2023rank}. Despite its flexibility that allows for an easy adaptation to particular problems, it requires a computationally expensive and sequential rank-$2r$ update step, which was later improved by the parallel integrator~\cite{CeKL23} to reduce costs by avoiding serial coefficient updates.

BUG integrators can be extended to higher orders by including high-order information in the basis. The augmented BUG integrator has been extended to second order in~\cite{ceruti2024robust}, providing a robust second-order error bound while preserving key structural properties. The increased order comes at the cost of necessitating the computation of projections of rank-$3r$ (or $4r$ for improved accuracy). As the authors of \cite{ceruti2024robust} point out, the core idea of their work can be used to derive further high-order versions of the BUG integrator. Such an extension has been proposed in \cite{nobile2025robust}. This extension relies on a coefficient update of the underlying $s$-stage Runge--Kutta method and thus limits the integrator to a specific explicit time integration scheme. This limits the integrator's use for stiff problems that frequently arise in scientific computing and thus suffers the same shortcomings as projection--based integrators such as \cite{KiV19}. Moreover, the integrator requires projections at rank $2rs$, which can become prohibitively expensive in most applications when $s$ is too large. An extension of the parallel integrator to second order has been proposed in \cite{kusch2024second}, which only requires the computation of projections at a rank of $2r$, similar to the first-order augmented BUG integrator.

In this work, we present high-order BUG integrators that are not limited by the aforementioned drawbacks. The main novelties of our work are:
\begin{itemize}
    \item[(i)] We derive novel BUG integrators that provably achieve order $p$ convergence under standard assumptions on the right-hand side. Unlike existing general order approaches \cite{KiV19,nobile2025robust}, these methods are not constrained by the explicit time integration scheme used for the basis update, enabling the use of general (e.g., implicit or exponential) time integration methods in the coefficient update. 
    \item[(ii)] The flexibility of using implicit coefficient updates allows the construction of efficient high-order integrators with augmentations following explicit low-storage schemes, significantly reducing the number of bases needed to achieve order $p$. As a result, implicit or exponential integrators can be employed in the coefficient updates with only a small number of required bases, rendering them computationally feasible. As an example, we derive an integrator for stiff problems of order six which only requires an implicit update for coefficient matrices of size $5r \times 5r$. 
    \item[(iii)] We derive conditions on the basis construction that allow for different strategies to obtain high-order BUG integrators. We focus on a Runge--Kutta-based strategy that, compared to \cite{nobile2025robust}, reduces the number of required augmentations by a factor of two for most time integration methods. 
    \item[(iv)] We validate our methods for stiff and non-stiff problems, showing improved efficiency and underlining that BUG integrators are not explicit methods, but can be used in combination with implicit schemes. 
\end{itemize}
 {The rest of the paper is structured as follows. First, we review arbitrary order DLRA integrators in Section~\ref{sec:RK-BUG}. In Section~\ref{sec:high_order_bug}, we present the construction of arbitrary order BUG integrators together with a robust error bound. Sections~\ref{sec:error_bound_abug} and \ref{sec:error_bound_pbug} provide the proofs of this error bound for the augmented BUG and the parallel BUG integrator. Implementation details of both integrators are provided in Sections~\ref{sec:impl_bugp} and \ref{sec:impl_parallel_bugp}. Lastly, we provide numerical examples in Section~\ref{sec:numerics} followed by a conclusion in Section~\ref{sec:summary_conculsion}. To facilitate the overview of notation used throughout the manuscript, we have added a notation table \ref{tab:notation} in the appendix. }

\section{Low-rank Runge-Kutta integrators}\label{sec:RK-BUG}

A first DLRA integrator that achieves general order is the projected Runge--Kutta method proposed in \cite{KiV19}. This integrator uses a Butcher Tableau with coefficients $a_{\ell, \ell'}, c_{\ell}$, and $b_{\ell}$, where $1 \leq \ell$ and $\ell' \leq \ell$. Denote by $[\![\bfZ]\!]$ the truncation of $\bfZ \in\mathbb{R}^{m\times n}$ to rank $r$. Then, the projected Runge--Kutta stages of an explicit $s$-stage Runge--Kutta method are defined as 
\begin{subequations}\label{eq:projRK}
\begin{align}
    t_{(\ell)} &= t_k + c_{\ell} h\,,
    \\
    \bfY_{(\ell)} 
    &= [\![\bfY_k + h\sum_{\ell' = 1}^{\ell-1}a_{\ell,\ell'}\bar \bfZ_{\ell'}]\!] 
    =:\bfU_{(\ell)}\bfS_{(\ell)}\bfV_{(\ell)}^{\top}\,,\label{eq:projRK_rksub}
    \\
    \bfF_{(\ell)} &= \bfF(t_{(\ell)}, \bfY_{(\ell)}) \,,
    \\
    \bar\bfZ_{\ell} &= \bfP(\bfY_{(\ell)})\bfF_{(\ell)}\,.
\end{align}
Here $\bar\bfZ_{\ell}$ denotes the projected low-rank intermediate Runge-Kutta solution at stage $\ell$.
\end{subequations}
The value of $\bfY_{k+1}$ is then assembled from these stages as
\begin{equation}
    \bfY_{k+1} = \bfY_{k} + h \sum_{\ell=1}^s b_\ell \bar\bfZ_{\ell}\,.
\end{equation}
This formulation is inefficient as it requires computing full matrices of the Runge-Kutta substeps $\bfY_{(\ell)}$ and the solutions $\bfY_{k}$ and $\bfY_{k+1}$. However, a common approach to ensure that all above steps can be performed in a low-rank format is to note that $\bfY_{(\ell)}$ is spanned by 
\begin{subequations}
    \begin{align*}
        \widehat\bfU_{(\ell)}^{\mathrm{RK}} :=\,& \text{orth}([\bfU_0, \,a_{\ell,1}\bfF_{(1)}\bfV_{(1)}, \,\cdots \,,\,a_{\ell,\ell-1}\bfU_{(\ell-1)}, \,a_{\ell,\ell-1}\bfF_{(\ell-1)}\bfV_{(\ell-1)}])\,,\\
        \widehat\bfV_{(\ell)}^{\mathrm{RK}} :=\,& \text{orth}([\bfV_0, \,a_{\ell,1}\bfF_{(1)}^{\top}\bfU_{(1)}, \,\cdots \,,\,a_{\ell,\ell-1}\bfV_{(\ell-1)}, \,a_{\ell,\ell-1}\bfF_{(\ell-1)}^{\top}\bfU_{(\ell-1)}])\,,
    \end{align*}
\end{subequations}
due to the structure of the tangent space which ensures that the range of $\bfU_{(\ell')}$ and $\bfF_{(\ell')}\bfV_{(\ell')}$ spans $\bar \bfZ_{\ell'}$ in \eqref{eq:projRK_rksub}. Further, note that if $a_{\ell,k} \equiv 0$, the corresponding entries can be left out. Thus, the corresponding expansion coefficients can be obtained through
\begin{align*}
    \widehat\bfS_{(\ell)}^{\mathrm{RK}} = \widehat\bfU_{(\ell)}^{\mathrm{RK},\top} \left( \bfY_k + h\sum_{\ell' = 1}^{\ell-1}a_{\ell,\ell'}\bar \bfZ_{\ell'}\right) \widehat\bfV_{(\ell)}^{\mathrm{RK}}\,.
\end{align*}
 {The update of $\bfY_{k+1}$ can be computed efficiently through the same strategy.} This integrator suffers two main drawbacks. First, and perhaps most importantly, the integrator is inherently explicit. I.e., allowable time step sizes must be chosen according to an explicit time step restriction, limiting its use in a large variety of applications. Second, the integrator violates underlying solution structures such as conservation. BUG integrators have been proposed to tackle these two challenges, where common BUG integrators can be used with general (e.g., implicit or exponential) time integration methods to enable their use in stiff settings \cite{CeL22,CeKL22,CeKL23,ceruti2024robust,kusch2024second} while augmented BUG integrators posses several conservation properties, e.g., for Schr\"odinger equations or gradient flows \cite{CeKL22} as well as local conservation laws for kinetic problems \cite{einkemmer2023conservation}. Note that high-order BUG integrators have been proposed in \cite{ceruti2024robust,kusch2024second}  {(see Appendix~ \ref{app:recap})}, which were later extended to ensure conservation for kinetic problems in \cite{einkemmer2023conservation}. A recently proposed BUG-type integrator that achieves general order as well as conservation for kinetic equations is the Runge--Kutta BUG integrator of \cite{nobile2025robust}. To achieve general order, it uses basis matrices 
\begin{subequations}
    \begin{align}
        \widehat\bfU_{\mathrm{RK}} :=\,& \text{orth}([\bfU_0, \,\beta_1\bfF_{(1)}\bfV_{(1)},\,\beta_2\bfU_{(2)}, \,\beta_2\bfF_{(2)}\bfV_{(2)}, \,\cdots \,,\,\beta_s\bfU_{(s)},\,\beta_s\bfF_{(s)}\bfV_{(s)}])\,,\\
        \widehat\bfV_{\mathrm{RK}} :=\,& \text{orth}([\bfV_0, \,\beta_1\bfF_{(1)}^{\top}\bfU_{(1)}, \,\beta_2\bfV_{(2)},\,\beta_2\bfF_{(2)}^{\top}\bfU_{(2)}, \,\cdots\,, \,\beta_s\bfV_{(s)}, \,\beta_s\bfF_{(s)}^{\top}\bfU_{(s)}])\,,
    \end{align}
\end{subequations}
where $\beta_i = 1$ if $b_i \neq 0$ and $\beta_i = 0$ if $b_i = 0$. Thus the method proposed in \cite{nobile2025robust} requires $2rs_*$ basis vectors, where $s_* \leq s$ is the number of non-zero values of $\beta_i$. For many conventional Runge--Kutta methods $ s_*=s$.  However, there are methods for which $s_* < s$.  For example, $s_* = s-1$ for Heun's third-order method and for the explicit midpoint method, while $s_*=s-6$ for Ketcheson's $10$-stage, $4\textsuperscript{th}$-order SSP-RK method \cite{Ketcheson2008}. 
The $S$-step is then defined as the coefficient update of the $s$-stage Runge--Kutta method
\begin{align}
    \wh \bfS_{\mathrm{RK}} = \widehat\bfU_{\mathrm{RK}}^{\top}\left(\bfY_0 + h\sum_{\ell = 1}^{s}b_{\ell}\bfF_{(\ell)}\right)\widehat\bfV_{\mathrm{RK}}\,.
\end{align}
{While this method is formally of order $p$, the large number of bases vectors (or columns) in $ \widehat\bfU_{\mathrm{RK}}$ and $ \widehat\bfV_{\mathrm{RK}}$ adds significant computational costs. Furthermore, since the final coefficient update of Runge--Kutta BUG integrators is a projection of the underlying explicit projected Runge--Kutta method, they are not applicable in stiff scenarios that require specialized (e.g., implicit or exponential) time integration methods. Thus, they violate one key advantage of classical BUG integrators.}

\section{High-order BUG integrators} \label{sec:high_order_bug}
In the following, we extend both the parallel \cite{CeKL23} and the augmented BUG integrators \cite{ceruti2023rank} to arbitrary order for a matrix ODE \eqref{eq:origProb}. Compared to existing low-rank Runge-Kutta integrators \cite{KiV19,nobile2025robust}, the proposed methods do not rely on an underlying Runge-Kutta method and are thus applicable to stiff problems. Moreover, compared to \cite{nobile2025robust}, they require smaller basis matrices. The integrators derived are natural extensions of the midpoint BUG integrator \cite{ceruti2024robust} and second-order parallel integrator \cite{kusch2024second} in that they generalize to those methods for order $p=2$. 
To provide analytic insight into the construction of these integrators, we make the standard assumptions that 
\begin{assumptions}\label{as:assumptions}
The right-hand side and the initial condition fulfill the following properties for the Frobenius norm $\Vert\cdot \Vert$:
\begin{enumerate}[label=A\arabic*, ref=A\arabic*]
	\item $\bfF$ is Lipschitz-continuous, i.e., for all $\bfZ, \widetilde{\bfZ} \in \mathbb{R}^{m \times n}$ and $0\le t \le T$,
    $$
    \| \bfF(t, \bfZ) - \bfF(t, \widetilde{\bfZ}) \| 
    \leq
    L \| \bfZ - \widetilde{\bfZ} \|.
    $$	\label{ass:Lipschitz}
    \item $\bfF$ is bounded, i.e., for all $\bfZ \in \mathbb{R}^{m \times n}$ and $0\le t \le T$, it holds $\| \bfF(t, \bfZ) \| \leq B$.\label{ass:bounded}
    \item Let $\mathcal{M}_r$ be the manifold of rank $r$ matrices, where $r\in\mathbb{N}$ is arbitrary. 
    For $\bfY\in\mathcal{M}_r$, the normal part of $\bfF(t, \bfY)$, that is $\bfR(t, \bfY) := (\bfI - 
    \bfP(\bfY))\bfF(t, \bfY)$ is $\varepsilon_r$-small, i.e., $\| \bfR(t, \bfY) \| \leq \varepsilon_r$. 
    \label{ass:epsilon}
    \item
    The error in the initial value is $\delta$-small, i.e., $\| \bfY_0 - \bfA_0 \| \le \delta$.
     \label{ass:IC}
\end{enumerate}	
\end{assumptions}

\subsection{Augmented BUG integrators}\label{sec:BUG-p-intro} To construct a general-order augmented BUG integrator, we first design the bases $\wt \bfU, \wt \bfV$ which are then used in a sequential Galerkin step to update expansion coefficients. In the following, we assume a time interval $t\in[t_n,t_{n+1}]$ and a numerical solution $\bfY_n = \bfU_n\bfS_n\bfV_n^{\top}$. {Two properties are important in the construction of basis matrices for the BUG method. 
\begin{enumerate}[label=P\arabic*, ref=P\arabic*]
    \item The basis spans the solution at the previous time step, i.e., $\wt \bfU\wt \bfU^{\top} \bfY_n \wt \bfV\wt \bfV^{\top} = \bfY_n$. \label{pr:ic_basis}
    \item For $\dt\bfA = \bfF(t,\bfA(t))$ with initial condition $\bfA(t_n) = \bfY_n$ the basis fulfills the consistency relation \label{pr:consistency_basis}
\begin{align}\label{eq:consistency_relation}
    \wt \bfU\wt \bfU^{\top} \bfA(t_{n+1}) \wt \bfV\wt \bfV^{\top} = \bfA(t_{n+1}) + O(h^{p+1} + h\varepsilon_r)\,.
\end{align}
\end{enumerate}
Given any basis fulfilling the above properties, the order $p$ augmented BUG integrator then solves
		\begin{equation}\label{bar-S-step} 
		\dt{\wt\bfS}(t) =  \wt\bfU^\top \bfF(t, \wt\bfU \, \wt\bfS(t) \wt\bfV^\top) \wt\bfV, 
		\qquad \wt\bfS(t_n) = \wt\bfM \bfS_n \wt\bfN^\top\,,
		\end{equation}
where $\wt \bfM=\wt \bfU^\top \bfU_n$ and $\wt\bfN=\wt \bfV^\top \bfV_n$. The time-updated solution is then given as $\bfY_{n+1} = [\![ \wt \bfU\bfS(t_{n+1})\wt \bfV^{\top}]\!]_{\vartheta}$, where $[\![ \cdot]\!]_{\vartheta}$ denotes the truncation with respect to a user-determined tolerance parameter $\vartheta$, following \cite{ceruti2023rank}. A detailed implementation of the augmented BUG integrator can be found in Section~\ref{sec:impl_bugp}. While Property~\ref{pr:ic_basis} is especially relevant for conservation, Property~\ref{pr:consistency_basis} is required to derive a robust error bound of the integrator. The consistency relation \ref{pr:consistency_basis} can be achieved through a large number of design choices. In the following, we focus on one of these.
} 
Given the Runge-Kutta substeps \eqref{eq:projRK}, we can define
\begin{subequations}\label{eq:augBasisGeneral}
\begin{align}
    \widehat\bfU_{(j)} :=\,& \text{orth}([\bfU_0, \,\bfF_{(1)}\bfV_{(1)}, \,\bfF_{(2)}\bfV_{(2)}, \,\cdots \,, \,\bfF_{(j)}\bfV_{(j)}])\,,\label{eq:augBasisGeneralU} \\ 
    \widehat\bfV_{(j)} :=\,& \text{orth}([\bfV_0, \,\bfF_{(1)}^{\top}\bfU_{(1)}, \,\bfF_{(2)}^{\top}\bfU_{(2)}, \,\cdots\,, \,\bfF_{(j)}^{\top}\bfU_{(j)}])\,. \label{eq:augBasisGeneralV}
\end{align}
\end{subequations}
Then, assume that an $s$-stage explicit Runge--Kutta method with order $p$ exists and has Butcher coefficients $c_i, b_i$, and $a_{ij}$ where $i,j\in\{1,\cdots, s\}$. An order $p$ augmented BUG integrator can be constructed with the bases $\wt\bfU := \wh \bfU_{(s)}$ and $\wt\bfV := \wh \bfV_{(s)}$, where $p=s=2$ resembles the rank-$4r$ BUG integrator based on the trapezoidal rule \cite[Remark~2]{ceruti2024robust}. We wish to mention that such an extension to general order $p$ was hinted at in the same remark. Compared to the RK-BUG integrator, the basis information $\beta_{\ell}\bfU_{(\ell)}$ and $\beta_{\ell}\bfV_{(\ell)}$ is not needed in our definition of the augmented bases, thus leading to a basis of rank $(s+1)r$. The reason for this is given in the following lemma:
\begin{lemma}[Basis information]\label{le:approxU}
    $\bfU_{(\ell)}$ is spanned by $\widehat \bfU_{(j-1)}$ and $\bfV_{(\ell)}$ is spanned by $\widehat \bfV_{(j-1)}$ for $\ell \leq j$.
\end{lemma}
\begin{proof}
We use induction to show this result holds for general $j\in\mathbb{N}$. The statement is true for $j=1$ since with $\widehat \bfU_{(0)} \equiv \bfU_{0}$ we have
\begin{align*}
    \bfU_{(1)} = \bfU_{0} \in \text{span}(\widehat \bfU_{(0)})\,.
\end{align*}
Let the statement hold for an arbitrary but fixed $j$ (IH). Then, since
\begin{align}
    \bfY_{(j+1)}:= \bfU_{(j+1)}\bfS_{(j+1)}\bfV_{(j+1)}^{\top} := [\![\bfY_0 + h\sum_{\ell = 1}^{j}a_{j,\ell}\bar \bfZ_{\ell}]\!]\,,
\end{align}
we know that 
\begin{align*}
    \bfU_{(j+1)} \in\,& \text{span}\{ \bfU_0, \bar \bfZ_1, \cdots,  \bar \bfZ_j\}\,\\
    =\,& \text{span}\{ \bfU_0, \bfF_{(1)}\bfV_{(1)},  \bfU_{(2)}, \bfF_{(2)}\bfV_{(2)}, \cdots,  \bfU_{(j)}, \bfF_{(j)}\bfV_{(j)}\}\,\\
    \stackrel{\mathrm{(IH)}}{=}\,& \text{span}\{ \widehat \bfU_{(j-1)}, \bfF_{(1)}\bfV_{(1)}, \cdots,  \bfF_{(j)}\bfV_{(j)}\}\\
    \stackrel{\mathrm{\eqref{eq:augBasisGeneralU}}}{=}\,& \text{span}\{ \widehat \bfU_{(j-1)}, \bfF_{(j)}\bfV_{(j)}\}\\
    \stackrel{\mathrm{\eqref{eq:augBasisGeneralU}}}{=}\,& \text{span}\{ \widehat \bfU_{(j)}\}\,.
\end{align*}
Showing that $\bfU_{(\ell)}\in \text{span}\{ \widehat \bfU_{(j)}\}$ for $\ell\leq j$ directly follows from the discussion above. The proof for $\bfV_{(\ell)}$ follows analogously.\qed
\end{proof}
Given this lemma, the consistency relation \eqref{eq:consistency_relation} directly follows:
\begin{lemma}
    Given Assumptions~\ref{as:assumptions} and $\bfA(t_n) = \bfY_n$, the bases $\wt\bfU := \wh \bfU_{(s)}$ and $\wt\bfV := \wh \bfV_{(s)}$ fulfill
    \begin{align*}
        \Vert\wt \bfU\wt \bfU^{\top} \bfA(t_{n+1}) \wt \bfV\wt \bfV^{\top} - \bfA(t_{n+1})\Vert \leq c_1 h^{p+1} + c_2 h\varepsilon_r\,,
    \end{align*}
    where constants only depend on the Lipschitz constant, the bound of $\bfF$, and a bound of {the p-th order} derivatives of the exact solution.
\end{lemma}
\begin{proof}
From \cite[Theorem~6]{KiV19} we know that the solution of the projected Runge--Kutta method $\bfY^{n+1}_{\mathrm{PRK}} = \bfY_n + h\sum_{i=1}^s b_i \bar \bfZ_i$ fulfills with $\mu=h(h^p+\eps_r)$ the local error bound
$$ 
\bfA(t_{n+1})-\bfY^{n+1}_{\mathrm{PRK}} = O(\mu).
$$
Due to Lemma~\ref{le:approxU}, the chosen augmented bases $\wt \bfU$ and $\wt \bfV$ span both $\bfU_{(i)}$, $\bfF_{(i)}\bfV_{(i)}$ and $\bfV_{(i)}$, $\bfF_{(i)}^\top\bfU_{(i)}$ for $i\leq s$. Hence,
$$
(\bfI-\wt \bfU \, \wt \bfU^\top) \bfY^{n+1}_{\mathrm{PRK}} = 0 \quad\hbox{ and }\quad 
\bfY^{n+1}_{\mathrm{PRK}} (\bfI-\wt \bfV \,\wt \bfV^\top) =0
$$
and thus $(\bfI-\wt \bfU \, \wt \bfU^\top)\bfA(t_{n+1}) =  (\bfI-\wt \bfU \, \wt \bfU^\top)(\bfA(t_{n+1})-\bfY^{n+1}_{\mathrm{PRK}}) = O(\mu)$. Analogously, we obtain $\bfA(t_{n+1})(\bfI-\wt \bfV \, \wt \bfV^\top) = O(\mu)$ and therefore
\begin{align*}
    \,&\bfA(t_{n+1}) - \wt \bfU \,\wt \bfU^\top \bfA(t_{n+1}) \wt \bfV\, \wt \bfV^\top\\
=\,& \bfA(t_{n+1}) - \wt \bfU \,\wt \bfU^\top \bfA(t_{n+1}) + \wt \bfU \,\wt \bfU^\top \bfA(t_{n+1}) - \wt \bfU \,\wt \bfU^\top \bfA(t_{n+1}) \wt \bfV\, \wt \bfV^\top = O(\mu)\,,
\end{align*}
which concludes the proof. \qed
\end{proof}
{
\begin{remark}
    The proof reveals the main strategy of constructing an order $p$ basis by spanning a (projected) order $p$ solution with sufficient accuracy, in this case $\bfY^{n+1}_{\mathrm{PRK}}$. Other approaches to construct a viable basis exist, e.g., using low-storage schemes as we do in Section~\ref{sec:num_stiff_heat}.
\end{remark}
}

\subsection{Parallel BUG integrators}
An order $p$ parallel integrator can be defined by choosing the pre-augmented bases as $\widehat \bfU_0 := \wh \bfU_{(s-1)}$ and $\wh \bfV_0 := \widehat \bfV_{(s-1)}$ if there exists an s-stage explicit Runge--Kutta method of order $p$ with Butcher coefficients $c_i, b_i$, and $a_{ij}$ where $i,j\in\{1,\cdots, s\}$. Note that this choice resembles the second-order parallel integrator \cite{kusch2024second} for $p = s = 2$. Thus, the parallel integrator requires $sr$ basis vectors.

 {The parallel integrator then solves (in parallel) for $t\in[t_n,t_{n+1}]$
\begin{alignat*}{2}
\label{eq:parallel_bugp_K_step}
    \dt{\bfK}(t)
    =\,&
    \bfF(t,\bfK(t)\widehat\bfV_0^\top)\widehat\bfV_0,
    \qquad
    &&\bfK(t_n)=\bfU_n\bfS_n\bfV_n^{\top}\wh\bfV_0\,,\\
    \dt{\bfL}(t)
    =\,&
    \bfF(t,\widehat\bfU_0\bfL(t)^\top)^\top\widehat\bfU_0,
    \qquad
    &&\bfL(t_n)=\bfV_n\bfS_n^{\top}\bfU_n^{\top}\wh\bfU_0\,,\\
    \dt{\bar{\bfS}}(t)
    =\,&
    \widehat\bfU_0^\top
    \bfF(t,\widehat\bfU_0\bar{\bfS}(t)\widehat\bfV_0^\top)
    \widehat\bfV_0,
    \qquad
    &&\bar{\bfS}(t_n)=\wh\bfV_0^{\top}\bfV_n\bfS_n\bfU_n^{\top}\wh\bfU_0\,.
\end{alignat*}
We then construct $\widehat \bfU= [\widehat \bfU_0, \widetilde \bfU_2] = \text{orth}([\widehat \bfU_0, \bfK(t_1)])$ and $\widehat \bfV= [\widehat \bfV_0, \widetilde \bfV_2] = \text{orth}([\widehat \bfV_0, \bfL(t_1)])$. 
Lastly, the coefficient matrix is assembled according to
\begin{equation}
\label{eq:parallel_bugp_assembled_S}
    \widehat\bfS_{n+1}^{\rm par}
    =
    \begin{pmatrix}
        \bar{\bfS}(t_{n+1})
        &
        \bfL(t_{n+1})^\top\widetilde\bfV_2
        \\
        \widetilde\bfU_2^\top\bfK(t_{n+1})
        &
        \bf0
    \end{pmatrix}.
\end{equation}
The time-updated solution is then given as $\bfY_{n+1}=[\![ \widehat\bfU\widehat\bfS_{n+1}^{\rm par}\widehat\bfV^\top]\!]_{\vartheta}$. A detailed implementation of the parallel BUG integrator can be found in Section~\ref{sec:impl_parallel_bugp}.}

\subsection{Robust error bound}
We can prove a robust error bound for both the parallel and augmented BUG integrators.
\begin{theorem}[Robust error bound]\label{th:robusterrorp} Under Assumptions~\ref{as:assumptions}, the error of the order $p$ parallel and augmented BUG integrators is bounded by
\begin{align}
    \| \bfY_k - \bfA(t_k)\|  \le C_0^{(p)} \delta + C_1^{(p)} h^p + C_2^{(p)} \eps_{r} + C_3^{(p)}  k\vartheta , \qquad 0\le kh \le T\,,
\end{align}
where constants only depend on the Lipschitz constant and bound of $\bfF$, a bound of {the p-th order} derivatives of the exact solution, and an upper bound of the time stepsize. 
\end{theorem}
 {We note that besides this robust error bound, both integrators can be made locally conservative for kinetic equations following \cite{einkemmer2023conservation}. In the case of the parallel integrator, Remark~\ref{re:local_cons_par} comments on how to remove the coefficient update, which directly yields a locally conservative scheme (when the conservative basis is included in the basis at time $t_0$ and a conservative truncation is used), see \cite[Section~7]{einkemmer2023conservation}. This scheme is also applicable for general time integration methods, i.e., can be used with implicit or exponential integrators. In the case of the augmented BUG integrator, local conservation can be achieved through additional basis augmentations, see \cite[Section~5]{einkemmer2023conservation}, when using explicit time integrators. Using an explicit Runge--Kutta method for the coefficient update of the form
\begin{align*}
    \wt \bfS^{n+1} = \wt \bfS^{n} + h \sum_{\ell=1}^{\hat s} b_{\ell} \widetilde \bfZ_{\ell}^{(S)}\,,
\end{align*}
local conservation can be achieved by additionally augmenting the bases $\wt\bfU$ with $\sum_{\ell=1}^{\hat s} b_{\ell} \widetilde \bfZ_{\ell}^{(S)}\bfV_{\mathrm{cons}}$, where $\bfV_{\mathrm{cons}}$ are the conservative basis functions, e.g., $\bfV_{\mathrm{cons}} = (1,\cdots,1)^{\top}$ when conserving mass in a nodal directional discretization for radiative transfer.
}

In the following two sections, we prove Theorem~\ref{th:robusterrorp} for the parallel and augmented BUG integrators. Throughout our derivations, all constants are independent of the curvature of the low-rank manifold. We often use the notation $\Vert \cdot \Vert \leq  O(\mu)$ to denote that the norm is bounded by terms that depend linearly on $\mu$ with constants independent of the curvature of the low-rank manifold. 

\section{Robust error bound for the high order augmented BUG integrator} \label{sec:error_bound_abug}

Given the basis properties \ref{pr:ic_basis} and \ref{pr:consistency_basis}, the robust error bound of Theorem~\ref{th:robusterrorp} for the augmented BUG integrator follows trivially along the lines of \cite[Theorem~2]{ceruti2024robust}. For the convenience of the reader, we repeat the derivation and adapt it to the order $p$ augmented BUG integrator in the following.
\begin{proof}[Theorem~\ref{th:robusterrorp} for the order $p$ augmented BUG integrator]
The proof is subdivided into three parts (a)--(c) in which we bound the local error. We assume $\bfA(t_0) = \bfY_0$ and obtain the global error bound with Lady Windermere's fan using the stability of the exact problem. \\
(a) By construction, the basis fulfills for $\mu=h(h^p+h\eps_r)$ the consistency relation \ref{pr:consistency_basis}
\begin{align*}
    \wt \bfU\wt \bfU^{\top} \bfA(t_{1}) \wt \bfV\wt \bfV^{\top} = \bfA(t_{1}) + O(\mu)\,.
\end{align*}
Therefore, since $\bfA(t_0)-\wt \bfU \,\wt \bfU^\top \bfA(t_0) \wt \bfV\wt \bfV^{\top} = 0$, we have
\begin{align*}
    \bfA(t_1) - \bfA(t_0) - \wt \bfU \,\wt \bfU^\top (\bfA(t_1)- \bfA(t_0)) \wt \bfV\, \wt \bfV^\top = O(\mu)\,.
\end{align*}
(b) For the residual term
$$
\bfR(t) = \bfA(t) - \wt \bfU \,\wt \bfU^\top \bfA(t) \wt \bfV\, \wt \bfV^\top,
$$
we therefore have
$$
\bfR(t_1)-\bfR(t_0) = O(\mu).
$$
By our local error analysis we have 
$$
\bfR(t_0) = 0  \quad\hbox{ and }\quad \bfR(t_1)=O(\mu).
$$
{Since $\bfR$ has bounded $p$-th derivatives, we have
\begin{align*}
    \Vert \bfR(t)\Vert =\,& \Vert \bfR(t_0) + \sum_{k=1}^{p-1}\bfR^{(k)}(t_0)\frac{(t-t_0)^k}{k!}\Vert + O(h^p)\\
    \leq\,& \Vert \sum_{k=1}^{p-1}\bfR^{(k)}(t_0)\frac{h^k}{k!}\Vert + O(h^p)\\
    \leq\,& \Vert \bfR(t_1)\Vert+ O(h^p) + O(h^p) = O(h^p +\mu)\,.
\end{align*}}
Thus we obtain
$$
\bfR(t) = O(h^p+\mu), \qquad t_0\le t \le t_1.
$$

(c) We write 
\begin{align*}
    \wt \bfY_1 - \bfA(t_1) =\,& \wt \bfY_1 - \wt \bfU\wt \bfU^\top \bfA(t_1) \wt \bfV \wt \bfV^\top + \wt \bfU\wt \bfU^\top \bfA(t_1) \wt \bfV \wt \bfV^\top - \bfA(t_1)\\
=\,&\wt \bfU (\wt \bfS(t_1) -  \wt \bfU^\top \bfA(t_1) \wt \bfV) \wt \bfV^\top - \bfR(t_1) 
\end{align*}
and show that
\begin{equation}\label{eq:S-tilde}
\wt \bfS(t_1) -  \wt \bfU^\top \bfA(t_1) \wt \bfV=O(\mu).
\end{equation}
Then, since by part (b), $\bfR(t_1) = O(\mu)$, we can conclude $\wt \bfY_1 - \bfA(t_1) = O(\mu)$. To show \eqref{eq:S-tilde}, let $t_0\le t \le t_1$, and let $ \widetilde\bfS(t) := \wt \bfU^\top \bfA(t)  \wt \bfV$.
		We write
		\begin{equation*}
		\begin{aligned}
		\bfA(t) 
		&=  \bigl( \bfA(t) - \wt \bfU \,\wt \bfU^\top \bfA(t) \wt \bfV \,\wt \bfV^\top \bigr) 
         + \wt \bfU \,\wt \bfU^\top \bfA(t) \wt \bfV \,\wt \bfV^\top 
		=  \textbf{R}(t)+ \wt \bfU \widetilde\bfS(t) \wt \bfV^\top
		\end{aligned}
		\end{equation*}
        and
        \begin{equation*}		\begin{aligned}
		\bfF(t, \bfA(t)) 
		&= \bfF(t, \wt \bfU \widetilde\bfS(t) \wt \bfV^\top + \textbf{R}(t) ) 
         = \bfF(t, \wt \bfU \widetilde\bfS(t) \wt \bfV^\top) + \bfD(t)
		\end{aligned}
		\end{equation*}
		where the defect $\bfD$ is given by
		$
		\bfD(t) := \bfF(t, \wt \bfU \widetilde\bfS(t) \wt \bfV^\top + \textbf{R}(t)) - \bfF(t, \wt \bfU \widetilde\bfS(t) \wt \bfV^\top).
		$
		Using the Lipschitz continuity of $\bfF$ with constant $L$ and the bound of $\bfR(t)$ from part (b), the defect is bounded by
		$$
		\|  \bfD(t) \| \le L \| \textbf{R}(t) \| =O(h^p+\mu).
		$$
		Lastly, we compare
		\begin{equation*}
		\begin{aligned}
		&\dt{\wt\bfS}(t) = \wt \bfU^\top \bfF(t, \wt \bfU\, \wt\bfS(t) \wt \bfV^\top)  \wt \bfV, 
		\qquad
		&\wt\bfS(t_0) = \wt \bfU^\top \bfY_0  \wt \bfV,
        \\
        &\dt{\widetilde\bfS}(t) = \wt \bfU^\top \bfF(t, \wt \bfU \widetilde\bfS(t) \wt \bfV^\top)  \wt \bfV + \wt \bfU^\top \bfD(t)  \wt \bfV, 
		\qquad
		&\widetilde\bfS(t_0) = \wt \bfU^\top \bfY_0  \wt \bfV.
		\end{aligned}
		\end{equation*}
		Then, the Gronwall inequality yields
		$$
		\|  \wt\bfS(t_1) - \widetilde \bfS(t_1) \| 
		\leq \int_{t_0}^{t_1} e^{L(t_1-s)} \, \| \bfD(s) \| \, ds
		=O(h(h^p+\mu))= O(\mu).
		$$
	
		Therefore, \eqref{eq:S-tilde} holds, and thus
        \begin{align*}
            \bfA(t_1) - \bfY_1 = \bfA(t_1) - \wt\bfY_1 + \wt\bfY_1 - \bfY_1 = O(\mu + \vartheta)
        \end{align*}
        which concludes the proof. 
\qed
\end{proof}

\begin{remark}\label{re:heps-truncation}
    As pointed out in \cite{nobile2025robust}, when the solution is truncated to the original rank $r$, then the $O(\vartheta)$ term can be replaced by a term depending on the normal components of the right-hand side. Even though an adaptive rank is often required in practical applications, we briefly repeat the arguments of \cite[Lemma~3]{KiV19} to achieve such a bound. When truncating to the original rank $r$ (or higher), we have
    \begin{align*}
        \Vert \bfY_1 - \wt\bfY_1\Vert =\,& \min_{\bfZ\in\mathcal{M}_{r}} \Vert \bfZ - \wt\bfY_1\Vert\\
        \leq\,& \min_{\bfZ\in\mathcal{M}_{r}} \Vert \bfZ - \bfA(t_1) \Vert + \Vert \bfA(t_1) - \wt\bfY_1\Vert.
    \end{align*}
    Thus, with $\bfX(t_1)$ being the solution of $\dt{\bfX}(t) = \bfP(\bfX(t))\bfF(t, \bfX(t))$ with initial condition $\bfX(t_0) = \bfY_0$, we know that
    \begin{align*}
        \Vert \bfY_1 - \wt\bfY_1\Vert \leq \Vert \bfX(t_1) - \bfA(t_1) \Vert + \Vert \bfA(t_1) - \wt\bfY_1\Vert = O(\mu)\,,
    \end{align*}
    where we use that $\Vert \bfX(t_1) - \bfA(t_1) \Vert = O(h\varepsilon_{r})$ according to, e.g., \cite[Lemma~1]{KiV19}.
\end{remark}

\section{Robust error bound for the high order parallel BUG integrator}\label{sec:error_bound_pbug}

To establish a robust error bound for the parallel integrator, a key step is to demonstrate that the constructed bases approximate both the full-rank and the projected Runge--Kutta substeps with sufficient accuracy. To this end, we define an $s$-stage Runge–Kutta method for the $K$-step of the parallel integrator as
\begin{subequations}\label{eq:RK-Kstep}
    \begin{align}
    \widetilde \bfZ_{j} =\,& \bfF(t_0 + c_{j}h, \bfY_0 + h\sum_{\ell = 1}^{j-1}a_{j,\ell}\widetilde \bfZ_{\ell})\widehat \bfV_0\widehat \bfV_0^{\top} \\
    \bfY_1^{(K)} :=\,& \bfK_1\widehat \bfV_0^{\top} = \bfY_0 + h\sum_{\ell = 1}^s b_{\ell}\widetilde \bfZ_{\ell}\,,
\end{align}
\end{subequations}
and for the $S$-step as
\begin{subequations}\label{eq:RK-Sstep}
    \begin{align}
    \widetilde \bfZ_{j}^{(S)} =\,& \widehat \bfU_0\widehat \bfU_0^{\top}\bfF(t_0 + c_{j}h, \bfY_0 + h\sum_{\ell = 1}^{j-1}a_{j,\ell}\widetilde \bfZ_{\ell}^{(S)})\widehat \bfV_0\widehat \bfV_0^{\top} \\
    \bfY_1^{(S)} :=\,& \widehat \bfU_0\bfS_1\widehat \bfV_0^{\top} = \bfY_0 + h\sum_{\ell = 1}^s b_{\ell}\widetilde \bfZ_{\ell}^{(S)}\,.
\end{align}
\end{subequations}
We wish to stress that such a time integration method is not required in the actual $K$ and $S$-step computations and serves the sole purpose of enabling our numerical analysis. In the following, we want to understand how well these time integration methods approximate the projected Runge--Kutta method \eqref{eq:projRK} and its full-rank counterpart
\begin{align*}
    \bfZ_{j} = \bfF(t_0 + c_{j}h, \bfY_0 + h\sum_{\ell = 1}^{j-1}a_{j,\ell}\bfZ_{\ell})\,,\quad \bfA_1 = \bfY_0 + h\sum_{\ell = 1}^s b_{\ell}\bfZ_{\ell}\,.
\end{align*}
We can show the following result
\begin{lemma}\label{le:rkstepsapprox}
Let Assumption~\ref{ass:epsilon} hold. Then, for $j\leq s-1$ we have that
    \begin{align}\label{eq:leRKclosenessS}
        \widetilde \bfZ_j = \bfZ_j + O(\varepsilon_r) = \bar \bfZ_j + O(\varepsilon_r)\,.
    \end{align}
\end{lemma}

\begin{proof}
This lemma is proven by induction. We have for $j=1$
\begin{align*}
    \widetilde \bfZ_1 = \bfF(t_0, \bfY_0)\widehat \bfV_0\widehat \bfV_0^{\top} \stackrel{\mathrm{Ass.}~\ref{ass:epsilon}}{=} \bfP(\bfY_0)\bfF(t_0, \bfY_0)\widehat \bfV_0\widehat \bfV_0^{\top} + O(\varepsilon_r).
\end{align*}
Since $\bfV_0$, $\bfF(t_0, \bfY_0)^{\top}\bfU_0$ and thus $\bfP(\bfY_0)\bfF(t_0, \bfY_0)$ are spanned by $\widehat \bfV_0$, we have $\widetilde \bfZ_1 = \bar \bfZ_1 + O(\varepsilon_r)$ and by boundedness of normal components $\bar \bfZ_1 = \bfZ_1 + O(\varepsilon_r)$.
Let us assume the induction hypothesis (IH) that the statement \eqref{eq:leRKcloseness} holds for an arbitrary but fixed $j-1 < s-1$. Then, for $j \leq s-1$
    \begin{align*}
        \widetilde \bfZ_j =\,& \bfF(t_0 + c_j h, \bfY_0 + h\sum_{\ell = 1}^{j-1}a_{j,\ell}\widetilde \bfZ_{\ell})\widehat \bfV_0\widehat \bfV_0^{\top}\\
        \stackrel{\mathrm{(IH),~\mathrm{ Ass.}~\ref{ass:Lipschitz}}}{=}\,& \bfF(t_0 + c_j h, \bfY_0 + h\sum_{\ell = 1}^{j-1}a_{j,\ell}\bar \bfZ_{\ell})\widehat \bfV_0\widehat \bfV_0^{\top} + O(Lh\varepsilon_r)\\
        =\,& \bfF_{(j)}\widehat \bfV_0\widehat \bfV_0^{\top} + O(Lh\varepsilon_r)\,.
    \end{align*}
    By boundedness of normal components, we have
    \begin{align*}
        \bfF_{(j)}\widehat \bfV_0\widehat \bfV_0^{\top} \stackrel{\text{Ass.}~\ref{ass:epsilon}}{=}\,&\bfP(\bfY_{(j)})\bfF_{(j)}\widehat \bfV_0\widehat \bfV_0^{\top} + O(\varepsilon_r)\,.
    \end{align*}    
    Moreover, by construction and with Lemma~\ref{le:approxU}, we know that $\widehat \bfV_0$ spans $\bfV_{(j)}$, $\bfF_{(j)}^{\top}\bfU_{(j)}$, and thus the terms $\bfP(\bfY_{(j)})\bfF_{(j)}$ for $j\leq s-1$ are spanned by $\widehat \bfV_0$ as well. Hence, 
    \begin{align*}
        \bfP(\bfY_{(j)})\bfF_{(j)}\widehat \bfV_0\widehat \bfV_0^{\top} =\,& \bfP(\bfY_{(j)})\bfF_{(j)} \stackrel{\text{Ass.}~\ref{ass:epsilon}}{=} \bfF_{(j)} + O(\varepsilon_r)\\
        =\,& \bfF(t_0 + c_j h, \bfY_0 + h\sum_{\ell = 1}^{j-1}a_{j,\ell} \bar \bfZ_{\ell}) + O(\varepsilon_r)\\
        \stackrel{\mathrm{(IH),~\mathrm{ Ass.}~\ref{ass:Lipschitz}}}{=}\,& \bfF(t_0 + c_j h, \bfY_0 + h\sum_{\ell = 1}^{j-1}a_{j,\ell} \bfZ_{\ell}) + O(\varepsilon_r) + O(Lh\varepsilon_r)\,.
    \end{align*}
    Thus, $\widetilde \bfZ_j = \bfZ_j + O(\varepsilon_r) = \bar\bfZ_j + O(\varepsilon_r)$ which concludes the proof. \qed
\end{proof}

The same result holds for the Runge-Kutta method of the $S$-step:
\begin{lemma}\label{le:rkstepsapproxS}
Let Assumption~\ref{ass:epsilon} hold. Then, for $j\leq s-1$ we have that
    \begin{align}\label{eq:leRKcloseness}
        \widetilde \bfZ_j^{(S)} = \bfZ_j + O(\varepsilon_r) = \bar \bfZ_j + O(\varepsilon_r)\,.
    \end{align}
\end{lemma}
\begin{proof}
    The proof follows analogously to that of Lemma~\ref{le:rkstepsapprox}.
\end{proof}

A key ingredient in proving the robust error bound is showing that the bases accurately span the full-rank solution. More precisely, we have for $\dt\bfA = \bfF(t,\bfA)$ with $\bfA(t_0) = \bfY_0$
\begin{lemma}\label{le:suff_acc_basis}
    Let Assumptions~\ref{ass:Lipschitz}, \ref{ass:bounded}, and \ref{ass:epsilon} hold. Then, the bases $\wh \bfU$, $\wh \bfV$ span $\bfA(t_1)$ sufficiently accurate, that is,
    \begin{align*}
        \Vert \bfA(t_1) - \wh \bfU\wh \bfU^{\top}\bfA(t_1)\wh \bfV\wh \bfV^{\top} \Vert \leq c_1h^{p+1} + c_2h\varepsilon_r\,.
    \end{align*}
\end{lemma}
\begin{proof}
We have 
\begin{align*}
    \Vert \bfA(t_1) - \wh \bfU\wh \bfU^{\top}\bfA(t_1)\wh \bfV\wh \bfV^{\top} \Vert =\,& \Vert \bfA(t_1) - \wh \bfU\wh \bfU^{\top}\bfA(t_1) + \wh \bfU\wh \bfU^{\top}\bfA(t_1)-\wh \bfU\wh \bfU^{\top}\bfA(t_1)\wh \bfV\wh \bfV^{\top} \Vert\\
    \leq\,& \Vert (\bfI - \wh \bfU\wh \bfU^{\top})\bfA(t_1)\Vert + \Vert \bfA(t_1)(\bfI-\wh \bfV\wh \bfV^{\top}) \Vert\,.
\end{align*}
We only bound the first term on the right-hand side and note that the second term follows analogously. Due to the order $p$ of the Runge--Kutta method $\bfA_1 = \bfY_0 + h\sum_{i=1}^s b_i \bfZ_i$ we have
\begin{align*}
    \Vert (\bfI - \wh \bfU\wh \bfU^{\top})\bfA(t_1)\Vert \leq \Vert (\bfI - \wh \bfU\wh \bfU^{\top})\bfA_1\Vert + O(h^{p+1})\,.
\end{align*}
By construction of the basis, we have $(\bfI-\wh \bfU\wh \bfU^{\top})\bfK(t_1) = 0$. Moreover, the Runge--Kutta method of the $K$-step \eqref{eq:RK-Kstep} is of order $p$ which implies $\bfY_1^{(K)} - \bfK(t_1)\wh \bfV_0^{\top} = O(h^{p+1})$. Taken together, we then have
\begin{align*}
    \Vert (\bfI - \wh \bfU\wh \bfU^{\top})\bfA_1\Vert =\,& \Vert (\bfI - \wh \bfU\wh \bfU^{\top})(\bfA_1-\bfK(t_1)\wh \bfV_0^{\top})\Vert\\
    \leq\,&\Vert (\bfI - \wh \bfU\wh \bfU^{\top})(\bfA_1-\bfY^{(K)}_1)\Vert + O(h^{p+1})\,.
\end{align*}
Comparing terms and using Lemma~\ref{le:rkstepsapprox} gives with $(\bfI-\wh \bfU\wh \bfU^{\top})\bfY_0 = 0$
\begin{align*}
    \Vert (\bfI-\wh \bfU\wh \bfU^{\top})(\bfA_1 -  \bfY^{(K)}_1) \Vert \leq\,& h\sum_{i=1}^s |b_i| \cdot \Vert (\bfI-\wh \bfU\wh \bfU^{\top})(\bfZ_i - \widetilde \bfZ_i) \Vert \\
    \stackrel{\mathrm{Le.~\ref{le:rkstepsapprox}}}{\leq}\,& h |b_s| \cdot \Vert (\bfI-\wh \bfU\wh \bfU^{\top})(\bfZ_s - \widetilde \bfZ_s) \Vert +O(h\varepsilon_r)\,.
\end{align*}
Then, with Lemma~\ref{le:rkstepsapprox} and Lipschitz continuity of $\bfF$
\begin{align*}
    \bfZ_s =\,& \bfF(t_0 + c_s h, \bfY_0 + h\sum_{\ell = 1}^{s-1}a_{s,\ell} \bfZ_{\ell}) = \bfF_{(s)} + O(\varepsilon_r)\,,\\
    \widetilde \bfZ_s =\,& \bfF(t_0 + c_s h, \bfY_0 + h\sum_{\ell = 1}^{s-1}a_{s,\ell} \widetilde \bfZ_{\ell})\wh \bfV_0\wh \bfV_0^{\top} = \bfF_{(s)}\wh \bfV_0\wh \bfV_0^{\top} + O(\varepsilon_r)\,.
\end{align*}
Thus, with the boundedness of normal components,
\begin{align}\label{eq:proofZsDiff}
    \Vert (\bfI-\wh \bfU\wh \bfU^{\top})(\bfZ_s - \widetilde \bfZ_s) \Vert \leq\,& \Vert (\bfI-\wh \bfU\wh \bfU^{\top})\bfF_{(s)}(\bfI - \wh \bfV_0\wh \bfV_0^{\top} ) \Vert + O(\varepsilon_r)\nonumber\\
    \leq\,& \Vert (\bfI-\wh \bfU\wh \bfU^{\top})\bfP(\bfY_{(s)})\bfF_{(s)}(\bfI - \wh \bfV_0\wh \bfV_0^{\top} ) \Vert + O(\varepsilon_r)\,.
\end{align}
From Lemma~\ref{le:approxU} we know that $\bfU_{(s)}$, $\bfV_{(s)}$ are spanned by $\wh \bfU_0$, $\wh\bfV_0$, respectively. Hence, since
\begin{align*}
    \bfP(\bfY_{(s)})\bfF_{(s)} = \bfU_{(s)}\bfU_{(s)}^{\top}\bfF_{(s)} + (\bfI - \bfU_{(s)}\bfU_{(s)}^{\top})\bfF_{(s)}\bfV_{(s)}\bfV_{(s)}^{\top}
\end{align*}
the norm in \eqref{eq:proofZsDiff} vanishes, which concludes the proof. \qed
\end{proof}

We are now able to prove Theorem~\ref{th:robusterrorp} for the parallel BUG integrator.
\begin{proof}[Theorem~\ref{th:robusterrorp} for the order $p$ parallel BUG integrator]
We assume $\bfA(t_0) = \bfY_0$ and obtain the global error bound with Lady Windermere's fan using the stability of the exact problem. 
    By the definition of the augmentation step and the triangular inequality, we have
\begin{align*}
    \Vert \widehat \bfU\widehat{\bfU}^{\top}\bfA(t_1)\widehat \bfV\widehat{\bfV}^{\top} -\wh\bfY_1 \Vert \leq\,& \Vert \widehat{\bfU}_0^{\top}\bfA(t_1)\widehat \bfV_0 - \bar \bfS(t_1) \Vert + \Vert \widetilde{\bfU}_2^{\top}\bfA(t_1)\widehat \bfV_0 - \widetilde{\bfU}_2^{\top}\bfK(t_1) \Vert \\
    \,&+ \Vert \widehat \bfU_0^{\top}\bfA(t_1)\widetilde{\bfV}_2 - \bfL(t_1)^{\top}\widetilde{\bfV}_2 \Vert + \Vert \widetilde{\bfU}_2^{\top}\bfA(t_1)\widetilde{\bfV}_2\Vert\,.
\end{align*}
In the following, we bound all terms on the right individually, starting with (a) the contribution of the $S$-step, (b) the contributions of the $K$ and $L$-steps and $(c)$ the term $\Vert \widetilde{\bfU}_2^{\top}\bfA(t_1)\widetilde{\bfV}_2\Vert$.\\
(a) We bound
\begin{align*}
    \Vert \wh\bfU_0^{\top}\bfA(t_1)\wh\bfV_0 - \bar\bfS(t_1) \Vert \leq \Vert \wh\bfU_0^{\top}\bfA_1\wh\bfV_0 - \bfS_1 \Vert + O(h^{p+1})\,,
\end{align*}
where $\bfA_1$ and $\bfS_1$ are the respective approximations with an order $p$ Runge-Kutta method. Similar to the proof of Lemma~\ref{le:suff_acc_basis}, we note that with Lemma~\ref{le:rkstepsapproxS} we have
\begin{align*}
    \Vert \wh\bfU_0^{\top}\bfA_1\wh\bfV_0 - \bfS_1 \Vert \leq\,& h\sum_{i=1}^s |b_i| \cdot \Vert \wh\bfU_0^{\top}(\bfZ_i - \widetilde \bfZ_i^{(S)})\wh\bfV_0 \Vert \\
    \stackrel{\mathrm{Le.~\ref{le:rkstepsapproxS}}}{\leq}\,& h |b_s| \cdot \Vert \wh\bfU_0^{\top}(\bfZ_s - \widetilde \bfZ_s^{(S)})\wh\bfV_0 \Vert +O(h\varepsilon_r)\,.
\end{align*}
Using Lemma~\ref{le:rkstepsapprox} and Lipschitz continuity of $\bfF$ yields
\begin{align*}
    \wh \bfU_0^{\top}\bfZ_s\wh \bfV_0 =\,& \wh \bfU_0^{\top}\bfF(t_0 + c_s h, \bfY_0 + h\sum_{\ell = 1}^{s-1}a_{s,\ell} \bfZ_{\ell})\wh \bfV_0 = \wh \bfU_0^{\top}\bfF_{(s)}\wh \bfV_0 + O(hL\varepsilon_r)\,,\\
    \wh \bfU_0^{\top}\widetilde \bfZ_s^{(S)}\wh \bfV_0 =\,& \wh \bfU_0^{\top}\bfF(t_0 + c_s h, \bfY_0 + h\sum_{\ell = 1}^{s-1}a_{s,\ell} \widetilde \bfZ_{\ell}^{(S)})\wh \bfV_0 = \wh \bfU_0^{\top}\bfF_{(s)}\wh \bfV_0 + O(hL\varepsilon_r)\,.
\end{align*}
Thus, $\Vert \wh\bfU_0^{\top}\bfA(t_1)\wh\bfV_0 - \bfS(t_1) \Vert = O(h^{p+1} + h\varepsilon_r)$.

(b) The bound of the second term follows analogously to that of (a). We begin with bounding
\begin{align*}
    \Vert \bfA(t_1)\wh\bfV_0 - \bfK(t_1) \Vert \leq \Vert \bfA_1\wh\bfV_0 - \bfK_1 \Vert + O(h^{p+1})\,,
\end{align*}
where $\bfA_1$ and $\bfK_1$ are the respective approximations with an order $p$ Runge-Kutta method as already used in the proof of Lemma~\ref{le:suff_acc_basis}. Similar to that proof, we note that with Lemma~\ref{le:rkstepsapprox} we have
\begin{align*}
    \Vert \bfA_1\wh\bfV_0 - \bfK_1 \Vert \leq\,& h\sum_{i=1}^s |b_i| \cdot \Vert (\bfZ_i - \widetilde \bfZ_i)\wh\bfV_0 \Vert \\
    \stackrel{\mathrm{Le.~\ref{le:rkstepsapprox}}}{\leq}\,& h |b_s| \cdot \Vert (\bfZ_s - \widetilde \bfZ_s)\wh\bfV_0 \Vert +O(h\varepsilon_r)\,.
\end{align*}
Using Lemma~\ref{le:rkstepsapprox} and Lipschitz continuity of $\bfF$ yields
\begin{align*}
    \bfZ_s\wh \bfV_0 =\,& \bfF(t_0 + c_s h, \bfY_0 + h\sum_{\ell = 1}^{s-1}a_{s,\ell} \bfZ_{\ell})\wh \bfV_0 = \bfF_{(s)}\wh \bfV_0 + O(hL\varepsilon_r)\,,\\
    \widetilde \bfZ_s\wh \bfV_0 =\,& \bfF(t_0 + c_s h, \bfY_0 + h\sum_{\ell = 1}^{s-1}a_{s,\ell} \widetilde \bfZ_{\ell})\wh \bfV_0 = \bfF_{(s)}\wh \bfV_0 + O(hL\varepsilon_r)\,.
\end{align*}
Thus, $\Vert \widetilde{\bfU}_2^{\top}\bfA(t_1)\wh\bfV_0 - \widetilde{\bfU}_2^{\top}\bfK(t_1) \Vert \leq \Vert\bfA(t_1)\wh\bfV_0 - \bfK(t_1) \Vert = O(h^{p+1} + h\varepsilon_r)$. Analogously, we can show $\Vert \wh\bfU_0^{\top}\bfA(t_1)\widetilde{\bfV}_2 - \bfL(t_1)^{\top}\widetilde{\bfV}_2 \Vert= O(h^{p+1} + h\varepsilon_r)$.

$(c)$ To bound the term $\Vert \widetilde{\bfU}_2^{\top}\bfA(t_1)\widetilde{\bfV}_2\Vert$ we again note that, according to \cite[Theorem~6]{KiV19}, the solution of the projected Runge--Kutta method $\bfY^1_{\mathrm{PRK}} = \bfY_0 + h\sum_{i=1}^s b_i \bar \bfZ_i$ fulfills with $\mu=h(h^p+h\eps_r)$
$$
\bfA(t_1)-\bfY^1_{\mathrm{PRK}} = O(\mu).
$$
Thus, with $\widetilde{\bfU}_2^{\top}\bfY_0\widetilde{\bfV}_2 = 0$, we have
\begin{align*}
    \Vert \widetilde{\bfU}_2^{\top}\bfA(t_1)\widetilde{\bfV}_2\Vert \leq\,& \sum_{i=1}^s hb_i \Vert \widetilde{\bfU}_2^{\top}\bar \bfZ_i\widetilde{\bfV}_2\Vert + O(\mu)\\
    \stackrel{\text{Le}.~\ref{le:rkstepsapprox}}{\leq}\,&\sum_{i=1}^{s-1} hb_i \Vert \widetilde{\bfU}_2^{\top}\widehat \bfZ_i\widetilde{\bfV}_2\Vert + hb_s \Vert \widetilde{\bfU}_2^{\top}\bar \bfZ_s\widetilde{\bfV}_2\Vert + O(\mu + h\varepsilon_r)\,.
\end{align*}
Since $\widehat \bfZ_i$ are spanned by $\wh \bfU_0$ for $i = 1,\cdots,s$, the sum term is zero, and we are left with bounding
\begin{align*}
    \Vert \widetilde{\bfU}_2^{\top}\bar \bfZ_s\widetilde{\bfV}_2\Vert \stackrel{\eqref{eq:projRK}}{=} \,&\Vert \widetilde{\bfU}_2^{\top} \bfP(\bfY_{(s)})\bfF_{(s)} \widetilde{\bfV}_2\Vert \\
    =\,& \Vert \widetilde{\bfU}_2^{\top} \big(\bfU_{(s)}\bfU_{(s)}^{\top}\bfF_{(s)}(\bfI - \bfV_{(s)}\bfV_{(s)}^{\top}) + \bfF_{(s)}\bfV_{(s)}\bfV_{(s)}^{\top}\big) \widetilde{\bfV}_2\Vert\,.
\end{align*}
By definition, $\wh \bfU_0 \equiv \wh \bfU_{(s-1)}$ and $\wh \bfV_0 \equiv \wh \bfV_{(s-1)}$ and, according to Lemma~\ref{le:approxU}, $\bfU_{(s)}$ is spanned by $\wh \bfU_{(s-1)}$ and $\bfV_{(s)}$ is spanned by $\wh \bfV_{(s-1)}$. Thus, $\widetilde{\bfU}_2^{\top} \bfU_{(s)} = \bfV_{(s)}^{\top} \widetilde{\bfV}_2 = 0$, and therefore $\Vert \widetilde{\bfU}_2^{\top}\bar \bfZ_s\widetilde{\bfV}_2\Vert = 0$. 

Therefore, with (a), (b), and (c), we conclude that $\Vert \widehat \bfU\widehat{\bfU}^{\top}\widetilde \bfA(t_1)\widehat \bfV\widehat{\bfV}^{\top} -\wh\bfY_1 \Vert = O(h(h^p + \varepsilon_r))$. We obtain the desired error bound with $\Vert \wh\bfY_1 - \bfY_1 \Vert \leq \vartheta$.\qed
\end{proof}

\begin{remark}\label{re:local_cons_par}
    By proving a bound for $\Vert\bfA(t_1)\wh\bfV_0 - \bfK(t_1) \Vert$, we have shown that the $S$-step of the parallel integrator can be replaced by $\bar\bfS_1 = \wh\bfU_0^{\top}\bfK(t_1)$, therefore only requiring two parallel rank $(s-1)r$ solves. Besides an improved efficiency, this yields a locally conservative integrator, following \cite{einkemmer2023conservation}.
\end{remark}

\section{Implementation details: BUG-$p$ integrator}
\label{sec:impl_bugp}

We describe one time step of the order $p$ augmented BUG integrator, denoted
BUG-$p$, from $t_n$ to $t_{n+1}=t_n+h$. As usual, let
$
    \bfY_n=\bfU_n\bfS_n\bfV_n^\top,\,
$$
    \bfU_n^\top\bfU_n=\bfI,
$  and  $
    \bfV_n^\top\bfV_n=\bfI .
$

The augmented basis is assembled with an RK basis constructor. This constructor
uses an explicit $s$-stage Runge--Kutta method of order $p$ only to generate the
basis information. The Galerkin $S$-step below is not required to use this
Runge--Kutta method.

Let $[\![\,\cdot\,]\!]_{r_n}$ denote the truncation to rank $r_n$, and let
$[\![\,\cdot\,]\!]_{\vartheta}$ denote the final truncation with tolerance
$\vartheta$.
The BUG-$p$ step consists of the following operations.

\begin{enumerate}
\item {\bf RK basis construction.}
Compute augmented orthonormal bases
\[
    (\wt\bfU,\wt\bfV)
    =
    {\tt RKBasisConstructor}
    (t_n,h,\bfU_n,\bfS_n,\bfV_n;\,a_{\ell j},c_\ell,s),
\]
where the coefficients $a_{\ell j}$ and $c_\ell$ are those of an explicit
$s$-stage Runge--Kutta method of order $p$, see Section \ref{sec_rk_basis_constr_bugp} for details.
It returns bases of the form
\[
\begin{aligned}
    \wt\bfU
    &=
    {\rm orth}
    ([\bfU_n,\bfF_{(1)}\bfV_{(1)},\ldots,\bfF_{(s)}\bfV_{(s)}]), \\
    \wt\bfV
    &=
    {\rm orth}
    ([\bfV_n,\bfF_{(1)}^\top\bfU_{(1)},\ldots,
      \bfF_{(s)}^\top\bfU_{(s)}]).
\end{aligned}
\]
Thus, without removing linear dependencies, the augmented bases have 
$\bar{r}\leq (s+1)r_n$ columns.  

\item {\bf Projection step.}
Project the current approximation into the augmented space using the $\bar{r}\times r_n$ matrices
$
    \wt\bfM=\wt\bfU^\top\bfU_n,
$
and
$
    \wt\bfN=\wt\bfV^\top\bfV_n,
$
\,i.e.,
\[
    \wt\bfS_n
    =
    \wt\bfU^\top\bfY_n\wt\bfV
    =
    \wt\bfM\bfS_n\wt\bfN^\top .
\]
Hence $\wt\bfU\wt\bfS_n\wt\bfV^\top$ is the orthogonal projection of
$\bfY_n$ onto the tensor-product space spanned by $\wt\bfU$ and $\wt\bfV$.

\item {\bf Galerkin $S$-step.}
Integrate the reduced matrix differential equation
\begin{equation}
\label{eq:bugp-S-step}
    \dt{\wt\bfS}(t)
    =
    \wt\bfU^\top
    \bfF(t,\wt\bfU\,\wt\bfS(t)\wt\bfV^\top)
    \wt\bfV,
    \qquad
    \wt\bfS(t_n)=\wt\bfS_n,
    \qquad
    t\in[t_n,t_{n+1}].
\end{equation}
The integrator used for \eqref{eq:bugp-S-step} is independent of the explicit
Runge--Kutta method used in the basis constructor. To retain order $p$, it is
sufficient to solve the reduced equation with a method whose local error is
$O(h^{p+1})$, for instance with a method of order at least $p$ for the reduced
problem. Thus, explicit methods may be used for non-stiff reduced problems,
whereas implicit, exponential, or otherwise stiffly stable methods can be used
when the reduced problem is stiff. If an order $q<p$ method is used, the overall
order is reduced accordingly.

Denote the computed coefficient matrix at $t_{n+1}$ by $\wt\bfS_{n+1}$ and set
\[
    \wt\bfY_{n+1}
    =
    \wt\bfU\,\wt\bfS_{n+1}\wt\bfV^\top .
\]

\item {\bf Truncation.}
Compress the augmented solution:
\[
    \bfY_{n+1}
    =
    \bfU_{n+1}\bfS_{n+1}\bfV_{n+1}^\top
    =
    [\![\wt\bfY_{n+1}]\!]_{\vartheta}.
\]
\end{enumerate}

\subsection{RK basis constructor.}\label{sec_rk_basis_constr_bugp}
We now describe the RK basis constructor used in Step~1. The constructor is most
naturally viewed as a sequence of projected BUG stage steps. Each stage performs
a basis augmentation, a Galerkin projection of the Runge--Kutta stage
combination, and a truncation back to the current rank. The resulting stage
value is then used to generate the next basis direction.

Set
$
    \bfY_{(1)}=\bfY_n
  $, $
    \bfU_{(1)}=\bfU_n  $, $
    \bfS_{(1)}=\bfS_n  $, and $
    \bfV_{(1)}=\bfV_n 
$.
For $\ell=1,\ldots,s$, let
$
    t_{(\ell)} = t_n+c_\ell h $, and $
    \bfF_{(\ell)}
    =
    \bfF(t_{(\ell)},\bfY_{(\ell)}).
$
Compute the two stage directions
\[
    \bfG_\ell=\bfF_{(\ell)}\bfV_{(\ell)},
    \qquad
    \bfH_\ell=\bfF_{(\ell)}^\top\bfU_{(\ell)} .
\]
The tangent-space projected increment is stored implicitly as
\begin{equation}
\label{eq:projected-stage-increment-factorized}
    \bar\bfZ_\ell
    =
    \bfP(\bfY_{(\ell)})\bfF_{(\ell)}
    =
    \bfU_{(\ell)}\bfH_\ell^\top
    +
    \bfG_\ell\bfV_{(\ell)}^\top
    -
    \bfU_{(\ell)}\bfC_\ell\bfV_{(\ell)}^\top,
    \qquad
    \bfC_\ell=\bfU_{(\ell)}^\top\bfG_\ell .
\end{equation}
Thus no full matrix representation of $\bar\bfZ_\ell$ is needed.

If $\ell<s$, the next Runge--Kutta stage is computed by a BUG-type stage step.

\begin{enumerate}
\item {\bf Stage basis augmentation.}
Construct
\[
\begin{aligned}
    \wh\bfU_{(\ell+1)}
    &=
    {\rm orth}
    ([\bfU_n,\bfG_1,\ldots,\bfG_\ell]), \\
    \wh\bfV_{(\ell+1)}
    &=
    {\rm orth}
    ([\bfV_n,\bfH_1,\ldots,\bfH_\ell]).
\end{aligned}
\]
Lemma~\ref{le:approxU} shows that the previously computed stage bases
$\bfU_{(j)}$ and $\bfV_{(j)}$, $j\le \ell$, are already contained in these
spans. Hence they do not have to be appended explicitly.

\item {\bf Stage Galerkin projection.}
Project the Runge--Kutta stage combination into the augmented stage space:
\begin{equation}
\label{eq:bugp-stage-projection}
    \wh\bfS_{(\ell+1)}
    =
    \wh\bfU_{(\ell+1)}^\top
    \left(
        \bfY_n
        +
        h\sum_{j=1}^{\ell}a_{\ell+1,j}\bar\bfZ_j
    \right)
    \wh\bfV_{(\ell+1)} .
\end{equation}
All terms in \eqref{eq:bugp-stage-projection} are evaluated from the low-rank
factors. In particular, the products with $\bar\bfZ_j$ use
\eqref{eq:projected-stage-increment-factorized}.

\item {\bf Stage truncation.}
Set
\[
    \wh\bfY_{(\ell+1)}
    =
    \wh\bfU_{(\ell+1)}
    \wh\bfS_{(\ell+1)}
    \wh\bfV_{(\ell+1)}^\top,
\]
and truncate to the current rank:
\[
    \bfY_{(\ell+1)}
    =
    \bfU_{(\ell+1)}
    \bfS_{(\ell+1)}
    \bfV_{(\ell+1)}^\top
    =
    [\![\wh\bfY_{(\ell+1)}]\!]_{r_n}.
\]
\end{enumerate}

After the final stage direction has been computed, the constructor returns
\begin{equation}
\label{eq:bugp-final-rk-basis}
\begin{aligned}
    \wt\bfU
    &=
    {\rm orth}
    ([\bfU_n,\bfG_1,\ldots,\bfG_s]), \\
    \wt\bfV
    &=
    {\rm orth}
    ([\bfV_n,\bfH_1,\ldots,\bfH_s]).
\end{aligned}
\end{equation}

\begin{algorithm}[t]
\caption{RK basis constructor for BUG-$p$}
\label{alg:rk_basis_constructor_bug_form}
\begin{algorithmic}[1]
\State \textbf{Input:}
$t_n,h,\bfU_n,\bfS_n,\bfV_n$; explicit RK coefficients $a_{\ell j},c_\ell$.  Number of RK stages $s$.
\State \textbf{Output:}
Augmented bases $\wt\bfU,\wt\bfV$.
\State
$\bfU_{(1)}\gets\bfU_n$,
$\bfS_{(1)}\gets\bfS_n$,
$\bfV_{(1)}\gets\bfV_n$.
\For{$\ell=1,\ldots,s$}
    \State
    $t_{(\ell)}\gets t_n+c_\ell h$.
    \State
    $\bfF_{(\ell)}
    \gets
    \bfF(t_{(\ell)},
    \bfU_{(\ell)}\bfS_{(\ell)}\bfV_{(\ell)}^\top)$.
    \State
    $\bfG_\ell\gets\bfF_{(\ell)}\bfV_{(\ell)}$,
    $\bfH_\ell\gets\bfF_{(\ell)}^\top\bfU_{(\ell)}$,
    $\bfC_\ell\gets\bfU_{(\ell)}^\top\bfG_\ell$.
    \If{$\ell<s$}
        \State
        $\wh\bfU_{(\ell+1)}
        \gets
        {\rm orth}([\bfU_n,\bfG_1,\ldots,\bfG_\ell])$.
        \State
        $\wh\bfV_{(\ell+1)}
        \gets
        {\rm orth}([\bfV_n,\bfH_1,\ldots,\bfH_\ell])$.
        \State
        Use  $
            \bar\bfZ_j
            =
            \bfU_{(j)}\bfH_j^\top
            +
            \bfG_j\bfV_{(j)}^\top
            -
            \bfU_{(j)}\bfC_j\bfV_{(j)}^\top
       $ to compute
        \[
            \wh\bfS_{(\ell+1)}
            =
            \wh\bfU_{(\ell+1)}^\top
            \left(
                \bfY_n
                +
                h\sum_{j=1}^{\ell}a_{\ell+1,j}\bar\bfZ_j
            \right)
            \wh\bfV_{(\ell+1)}
        \]

        \State
        $\bfY_{(\ell+1)}
        =
        \bfU_{(\ell+1)}
        \bfS_{(\ell+1)}
        \bfV_{(\ell+1)}^\top
        \gets
        [\![
        \wh\bfU_{(\ell+1)}
        \wh\bfS_{(\ell+1)}
        \wh\bfV_{(\ell+1)}^\top
        ]\!]_{r_n}$.
    \EndIf
\EndFor
\State
$\wt\bfU\gets{\rm orth}([\bfU_n,\bfG_1,\ldots,\bfG_s])$.
\State
$\wt\bfV\gets{\rm orth}([\bfV_n,\bfH_1,\ldots,\bfH_s])$.
\State
\Return $(\wt\bfU,\wt\bfV)$.
\end{algorithmic}
\end{algorithm}

\begin{remark}[Low-rank implementation]
All operations in the RK basis constructor can be performed in low-rank format,
provided that the right-hand side admits the required projected products. In
particular, the full matrices $\bfY_{(\ell)}$ and $\bar\bfZ_\ell$ do not have to
be formed. The constructor only requires products of the form
\[
    \bfF_{(\ell)}\bfV_{(\ell)},\qquad
    \bfF_{(\ell)}^\top\bfU_{(\ell)},
\]
together with small projected matrices such as
\[
    \wh\bfU_{(\ell+1)}^\top
    \bar\bfZ_j
    \wh\bfV_{(\ell+1)} .
\]

\end{remark}

\begin{remark}[Memory optimization]
The basis constructor  keeps all stage direction
blocks $\bfG_\ell=\bfF_{(\ell)}\bfV_{(\ell)}$ and
$\bfH_\ell=\bfF_{(\ell)}^\top\bfU_{(\ell)}$, $\ell=1,\ldots,s$. Thus it does
not exploit possible zero coefficients in the Butcher tableau. A smaller
constructor may be used if the retained blocks still satisfy the stage-span
property required in Lemma~\ref{le:approxU}. More precisely, whenever a retained
stage direction $\bar\bfZ_\ell$ is used, the corresponding bases
$\bfU_{(\ell)}$ and $\bfV_{(\ell)}$ must be contained in the retained augmented
spaces. This can be ensured either by retaining the preceding direction blocks
needed by Lemma~\ref{le:approxU}, or by explicitly including the retained bases
such as $\bfU_{(b)}$ and $\bfV_{(b)}$.

Consequently, zero entries $a_{\ell+1,j}=0$ may be used to omit the
corresponding increment from the temporary stage combination. Similarly, a stage
with $b_\ell=0$ may be omitted from the final basis only if it is not needed,
directly or indirectly, to span later retained stages. Otherwise its contribution
must be carried either through the direction blocks
$\bfG_\ell,\bfH_\ell$ or through suitable retained bases.
\end{remark}

\section{Implementation details: parallel BUG-$p$ integrator}
\label{sec:impl_parallel_bugp}

We describe one time step of the order $p$ parallel BUG integrator, denoted
parallel BUG-$p$, from $t_n$ to $t_{n+1}=t_n+h$. As usual, let
$
    \bfY_n=\bfU_n\bfS_n\bfV_n^\top $ with $
    \bfU_n^\top\bfU_n=\bfI,$ and $
    \bfV_n^\top\bfV_n=\bfI .
$
The method first constructs pre-augmented bases
$\widehat\bfU_0$ and $\widehat\bfV_0$ by an RK pre-basis constructor, analogously to Section \ref{sec_rk_basis_constr_bugp}, with the difference that the
constructor uses an explicit $s$-stage Runge--Kutta method of order $p$ only to
generate basis information and stops after the $(s-1)^{st}$ stage.
The parallel BUG-$p$ step consists of the following operations.

\begin{enumerate}
\item {\bf RK pre-basis construction.}
Compute pre-augmented orthonormal bases
\[
    (\widehat\bfU_0,\widehat\bfV_0)
    =
    {\tt RKPreBasisConstructor}
    (t_n,h,\bfU_n,\bfS_n,\bfV_n;\,a_{\ell j},c_\ell, s),
\]
where the coefficients $a_{\ell j}$ and $c_\ell$ are those of the explicit
$s$-stage Runge--Kutta method used for basis generation. The constructor returns
bases of the form
\[
\begin{aligned}
    \widehat\bfU_0
    &=
    {\rm orth}
    ([\bfU_n,\bfF_{(1)}\bfV_{(1)},\ldots,
      \bfF_{(s-1)}\bfV_{(s-1)}]), \\
    \widehat\bfV_0
    &=
    {\rm orth}
    ([\bfV_n,\bfF_{(1)}^\top\bfU_{(1)},\ldots,
      \bfF_{(s-1)}^\top\bfU_{(s-1)}]).
\end{aligned}
\]
Thus, without removing linear dependencies, the augmented bases have 
$\hat{r}\leq sr_n$ columns.  

\item {\bf Parallel $K$-, $L$-, and $S$-steps.}
The following three matrix differential equations are independent after
$\widehat\bfU_0$ and $\widehat\bfV_0$ have been constructed and can therefore be
solved in parallel.

The $K$-step is
\begin{equation}
\label{eq:parallel_bugp_K_step}
    \dt{\bfK}(t)
    =
    \bfF(t,\bfK(t)\widehat\bfV_0^\top)\widehat\bfV_0,
    \qquad
    \bfK(t_n)=\bfU_n\bfS_n\bfV_n^{\top}\wh\bfV_0,
    \qquad
    t\in[t_n,t_{n+1}].
\end{equation}
The $L$-step is
\begin{equation}
\label{eq:parallel_bugp_L_step}
    \dt{\bfL}(t)
    =
    \bfF(t,\widehat\bfU_0\bfL(t)^\top)^\top\widehat\bfU_0,
    \qquad
    \bfL(t_n)=\bfV_n\bfS_n^{\top}\bfU_n^{\top}\wh\bfU_0,
    \qquad
    t\in[t_n,t_{n+1}].
\end{equation}
The Galerkin $S$-step in the pre-augmented tensor-product space is
\begin{equation}
\label{eq:parallel_bugp_S_step}
    \dt{\bar{\bfS}}(t)
    =
    \widehat\bfU_0^\top
    \bfF(t,\widehat\bfU_0\bar{\bfS}(t)\widehat\bfV_0^\top)
    \widehat\bfV_0,
    \qquad
    \bar{\bfS}(t_n)=\wh\bfV_0^{\top}\bfV_n\bfS_n\bfU_n^{\top}\wh\bfU_0,
    \qquad
    t\in[t_n,t_{n+1}].
\end{equation}
To retain order $p$, it is sufficient to solve these subproblems
with any time integration method whose local errors are $O(h^{p+1})$. 
Denote the computed endpoint values by
\[
    \bfK_{n+1}\approx\bfK(t_{n+1}),\qquad
    \bfL_{n+1}\approx\bfL(t_{n+1}),\qquad
    \bar{\bfS}_{n+1}\approx\bar{\bfS}(t_{n+1}).
\]

\item {\bf Final basis augmentation.}
Extract the new basis information generated by the $K$- and $L$-steps:
\[
    \widetilde\bfU_2
    =
    {\rm orth}
    \bigl((\bfI-\widehat\bfU_0\widehat\bfU_0^\top)\bfK_{n+1}\bigr),
    \qquad
    \widetilde\bfV_2
    =
    {\rm orth}
    \bigl((\bfI-\widehat\bfV_0\widehat\bfV_0^\top)\bfL_{n+1}\bigr).
\]
If one of the projected blocks is numerically rank deficient, linearly
dependent columns are discarded. Set
$
    \widehat\bfU=[\widehat\bfU_0,\widetilde\bfU_2]$, and $
    \widehat\bfV=[\widehat\bfV_0,\widetilde\bfV_2]
$. 
The untruncated final representation may therefore have more columns than the
pre-augmented bases. However, the expensive projected time evolutions
\eqref{eq:parallel_bugp_K_step}--\eqref{eq:parallel_bugp_S_step}
are performed only at the pre-augmented ranks $\hat{r}\le sr_n$.

\item {\bf Assembly of the augmented coefficient matrix.}
Assemble the coefficient matrix in the final augmented bases as
\begin{equation}
\label{eq:parallel_bugp_assembled_S}
    \widehat\bfS_{n+1}^{\rm par}
    =
    \begin{pmatrix}
        \bar{\bfS}_{n+1}
        &
        \bfL_{n+1}^\top\widetilde\bfV_2
        \\
        \widetilde\bfU_2^\top\bfK_{n+1}
        &
        \bf0
    \end{pmatrix}.
\end{equation}
Thus the uncompressed endpoint approximation is
$
    \widehat\bfY_{n+1}^{\rm par}
    =
    \widehat\bfU\,
    \widehat\bfS_{n+1}^{\rm par}
    \widehat\bfV^\top .
$

\item {\bf Truncation.}
Compress the augmented solution:
\[
    \bfY_{n+1}
    =
    \bfU_{n+1}\bfS_{n+1}\bfV_{n+1}^\top
    =
    [\![\widehat\bfY_{n+1}^{\rm par}]\!]_{\vartheta}.
\]
\end{enumerate}

\begin{algorithm}[t]
\caption{Order $p$ parallel BUG integrator (parallel BUG-$p$)}
\label{alg:parallel_bugp}
\begin{algorithmic}[1]
\State \textbf{Input:}
$t_n,h,\bfU_n,\bfS_n,\bfV_n$; explicit RK coefficients
$a_{\ell j},c_\ell$ for the pre-basis constructor; order-$p$ solvers for the
$K$-, $L$-, and $S$-steps; truncation tolerance $\vartheta$. Number of RK stages $s$.
\State \textbf{Output:}
$\bfU_{n+1},\bfS_{n+1},\bfV_{n+1}$.

\Statex
\State
$(\widehat\bfU_0,\widehat\bfV_0)
\gets
{\tt RKPreBasisConstructor}
(t_n,h,\bfU_n,\bfS_n,\bfV_n;\,a_{\ell j},c_\ell,s)$.
\State
$\widehat\bfM\gets\widehat\bfU_0^\top\bfU_n$,
$\widehat\bfN\gets\widehat\bfV_0^\top\bfV_n$.

\Statex
\State
$\bfK_n\gets \bfU_n\bfS_n\widehat\bfN^\top$.
\State
$\bfL_n\gets \bfV_n\bfS_n^\top\widehat\bfM^\top$.
\State
$\bar{\bfS}_n\gets \widehat\bfM\bfS_n\widehat\bfN^\top$.

\Statex
\State
Compute $\bfK_{n+1}$ by solving
\[
    \dt{\bfK}(t)
    =
    \bfF(t,\bfK(t)\widehat\bfV_0^\top)\widehat\bfV_0,
    \qquad
    \bfK(t_n)=\bfK_n .
\]
\State
Compute $\bfL_{n+1}$ by solving
\[
    \dt{\bfL}(t)
    =
    \bfF(t,\widehat\bfU_0\bfL(t)^\top)^\top\widehat\bfU_0,
    \qquad
    \bfL(t_n)=\bfL_n .
\]
\State
Compute $\bar{\bfS}_{n+1}$ by solving
\[
    \dt{\bar{\bfS}}(t)
    =
    \widehat\bfU_0^\top
    \bfF(t,\widehat\bfU_0\bar{\bfS}(t)\widehat\bfV_0^\top)
    \widehat\bfV_0,
    \qquad
    \bar{\bfS}(t_n)=\bar{\bfS}_n .
\]

\Statex
\State
$\widetilde\bfU_2
\gets
{\rm orth}\bigl((\bfI-\widehat\bfU_0\widehat\bfU_0^\top)\bfK_{n+1}\bigr)$.
\State
$\widetilde\bfV_2
\gets
{\rm orth}\bigl((\bfI-\widehat\bfV_0\widehat\bfV_0^\top)\bfL_{n+1}\bigr)$.
\State
$\widehat\bfU\gets[\widehat\bfU_0,\widetilde\bfU_2]$,
$\widehat\bfV\gets[\widehat\bfV_0,\widetilde\bfV_2]$.

\Statex
\State
\[
    \widehat\bfS_{n+1}^{\rm par}
    \gets
    \begin{pmatrix}
        \bar{\bfS}_{n+1}
        &
        \bfL_{n+1}^\top\widetilde\bfV_2
        \\
        \widetilde\bfU_2^\top\bfK_{n+1}
        &
        \bf0
    \end{pmatrix}.
\]
\State
\[
    \bfY_{n+1}
    =
    \bfU_{n+1}\bfS_{n+1}\bfV_{n+1}^\top
    \gets
    [\![
        \widehat\bfU
        \widehat\bfS_{n+1}^{\rm par}
        \widehat\bfV^\top
    ]\!]_{\vartheta}.
\]
\State \Return $(\bfU_{n+1},\bfS_{n+1},\bfV_{n+1})$.
\end{algorithmic}
\end{algorithm}

The parallel integrator is especially relevant for tensor-valued differential equations where multiple factors need to be solved in parallel. However, it has some distinct advantages for matrix-valued problems as well: For GPU computations, truncation steps that must be computed in the basis constructor as well as after the time update become a main bottleneck \cite{boukaram2017batched}. In this scenario, the repeated truncations in the basis constructor can become infeasible, leading to a doubling of the rank in each step of the constructor. In that case, instead of reducing the basis by only $r$ vectors, the parallel integrator halves the number of basis vectors, thus leading to significantly smaller differential equations.

\section{Numerical results}\label{sec:numerics}

 {In the following, we present numerical experiments for the proposed integrators. All numerical results can be reproduced with the openly available Julia code \cite{code}.}

\subsection{Non-stiff Schr\"odinger equation}
\label{sec:num_nonstiff_schroedinger}

\begin{figure}[h]
    \centering

    \begin{subfigure}[t]{0.49\textwidth}
        \centering
        \includegraphics[width=\textwidth]{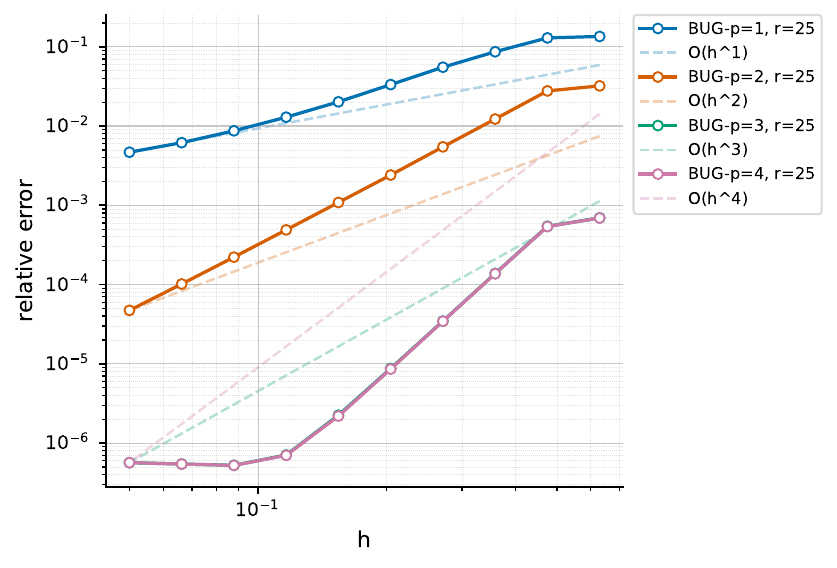}
        \caption{BUG-$p$, RK-8.}
        \label{fig:num_nonstiff_schroedinger_bugp}
    \end{subfigure}
    \hfill
    \begin{subfigure}[t]{0.49\textwidth}
        \centering
        \includegraphics[width=\textwidth]{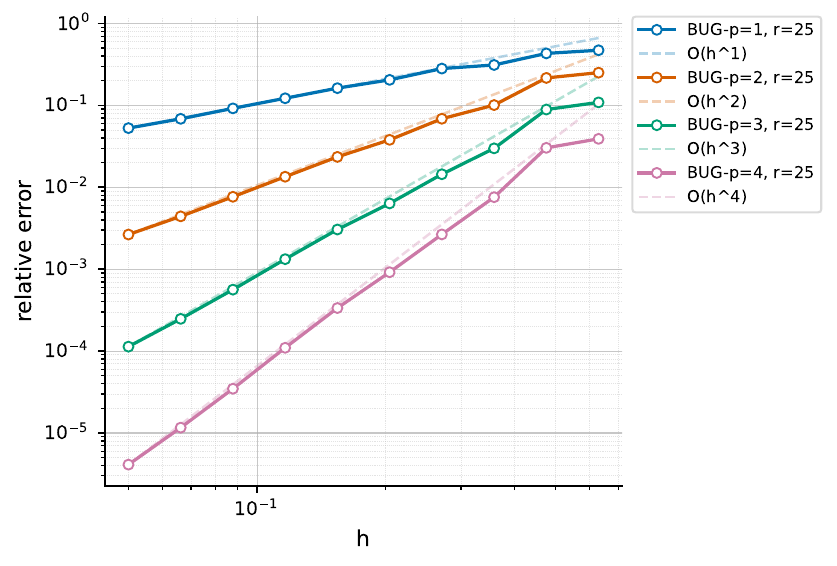}
        \caption{BUG-$p$, same order.}
        \label{fig:num_nonstiff_schroedinger_bugp_same_order}
    \end{subfigure}

    \vspace{0.5em}

    \begin{subfigure}[t]{0.49\textwidth}
        \centering
        \includegraphics[width=\textwidth]{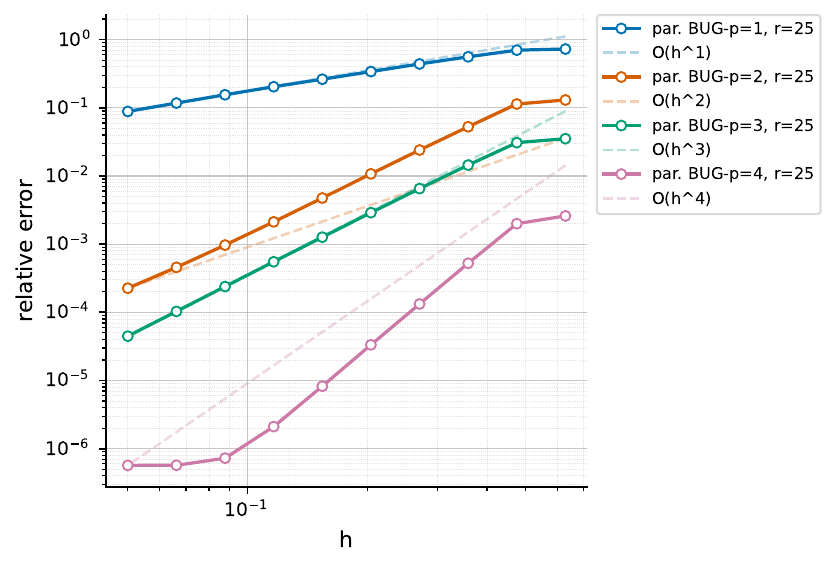}
        \caption{Parallel BUG-$p$, RK-8.}
        \label{fig:num_nonstiff_schroedinger_parallel_bugp}
    \end{subfigure}
    \hfill
    \begin{subfigure}[t]{0.49\textwidth}
        \centering
        \includegraphics[width=\textwidth]{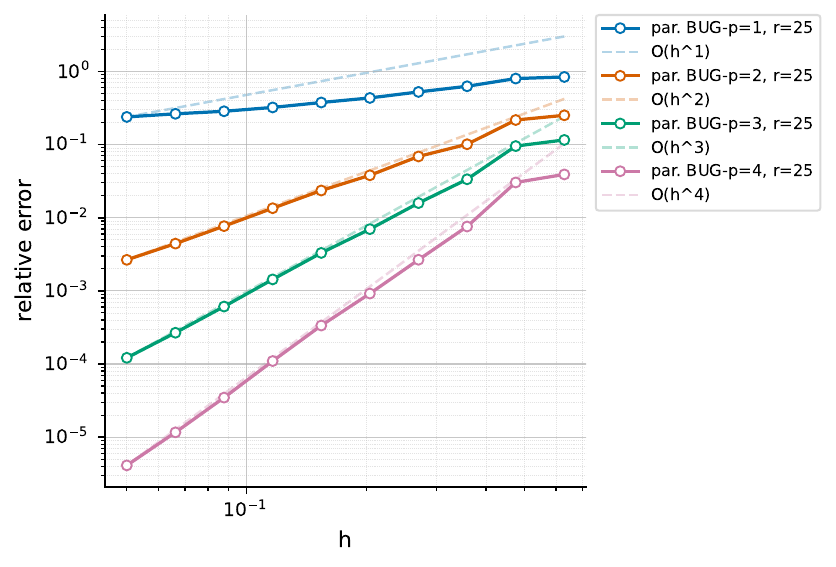}
        \caption{Parallel BUG-$p$, same order.}
        \label{fig:num_nonstiff_schroedinger_parallel_bugp_same_order}
    \end{subfigure}

    \vspace{0.5em}

    \begin{subfigure}[t]{0.49\textwidth}
        \centering
        \includegraphics[width=\textwidth]{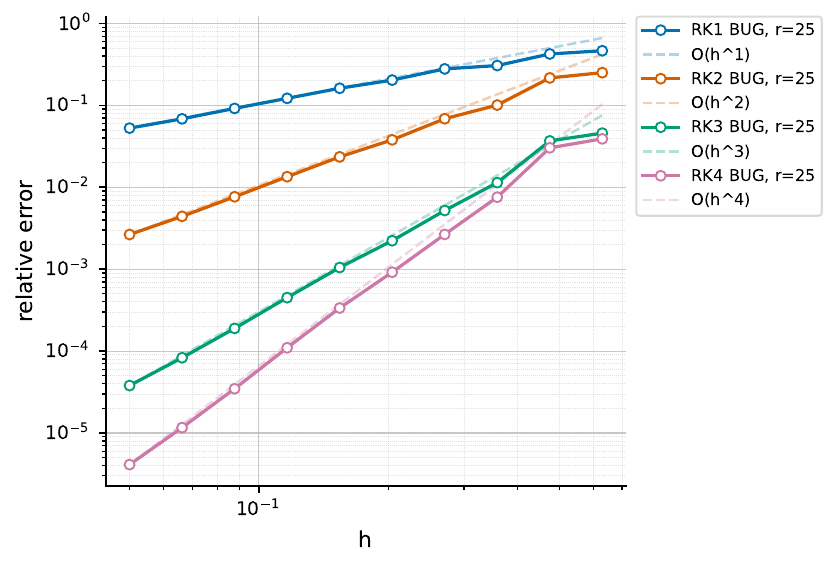}
        \caption{RK-BUG \cite{nobile2025robust}.}
        \label{fig:num_nonstiff_schroedinger_rkbug}
    \end{subfigure}

    \caption{Convergence for the non-stiff Schr\"odinger equation
    \eqref{eq:discr_schoedinger}. The panels show the
    final-time relative error for the proposed BUG-$p$ and parallel BUG-$p$
    methods together with their corresponding same-order variants, compared
    with the Runge--Kutta BUG method of \cite{nobile2025robust}. }
    \label{fig:num_nonstiff_schroedinger}
\end{figure}

\begin{figure}[t]
    \centering
    \includegraphics[width=0.49\linewidth]{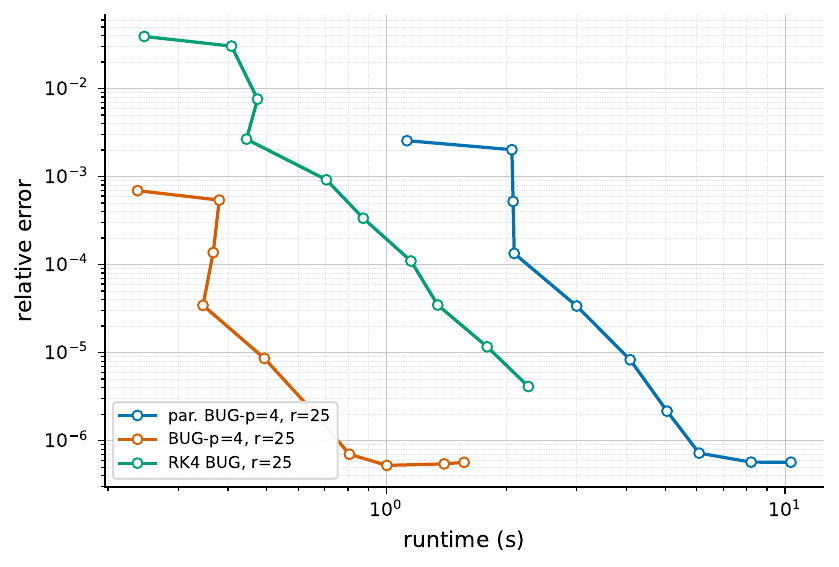}
    \includegraphics[width=0.49\linewidth]{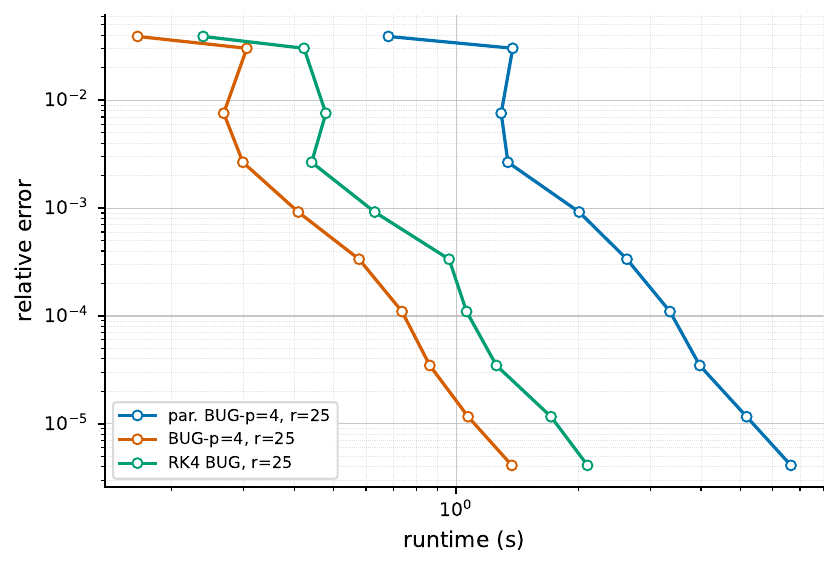}
    \caption{Efficiency of different low-rank integrators of order $p=4$ for the non-stiff Schrödinger equations \eqref{eq:discr_schoedinger}. Each marker denotes one time-step choice. Left: Evolution equations are solved with an order $8$ Runge--Kutta scheme. Right: Evolution equations are solved with a Runge--Kutta scheme of the same order as the BUG integrators. BUG-p has the best error-to-run-time inflection point, compared to the RK-BUG and parallel BUG-p in both scenarios. }
    \label{fig:efficiency_nonstiff_schroedinger}
\end{figure}

Following \cite{ceruti2024robust}, we first consider a non-stiff Schr\"odinger-type equation on the two-dimensional torus $\torus^2=[-\pi,\pi)^2$. The problem takes the form
\begin{align}
	   i\dot{\bfA}(t) = \bfH[\bfA(t)],
	\label{eq:discr_schoedinger}
\end{align}
where the right-hand side is given as
\begin{align*}
	&\bfH[\bfA] = -\frac{1}{2} \Big(\bfD \bfA + \bfA \bfD^\top \Big) + \bfV_\text{cos}\bfA\bfV_\text{cos} \in \R^{n \times n}\,,
	\\
	& \bfD = \texttt{tridiag}(-1,2,-1) + \mathbf{e}_1\mathbf{e}_n^{\top} + \mathbf{e}_n\mathbf{e}_1^{\top} \in \R^{n \times n} \,, 
	\\
	& \bfV_\text{cos} = \text{diag}(1-\cos(2\pi j / n)) \quad j = -n/2, \cdots, n/2 - 1 \,,
\end{align*}
with the $j$-th unit vector denoted as $\mathbf{e}_j$. This corresponds to the spatial discretization of the non-stiff Schrödinger equation using a standard tensor-product centered finite
difference discretization. The initial condition is chosen as
\[
    \bfA(0)=\bfU_0\bfS_0\bfV_0^*,
    \qquad
    \bfS_0=\diag(10^{-1},10^{-2},\ldots,10^{-N}),
\]
where $\bfU_0$ and $\bfV_0$ are random orthonormal matrices generated with a fixed random seed. This problem is evolved until a final time of $T=0.5$.

For the time integration, we compare the proposed augmented
and parallel BUG-$p$ integrators with the Runge--Kutta BUG method
of \cite{nobile2025robust}. Since the spatially discretized problem is
non-stiff, the reduced subproblems are solved with explicit time integrators. Here, we investigate the effects of using a matching Runge--Kutta method to solve the $K$, $L$, and $S$ equations (e.g., for a BUG-$p$ integrator we use an order $p$ Runge--Kutta method) as well as a fixed order of $8$ (RK-8). As remarked, the parallel BUG-$p$ integrator is beneficial when no truncation is computed in the basis construction. Thus, we do not compute singular value decompositions in the construction of the basis, leading to an increased rank for the $K$, $L$, and $S$-steps. Moreover, we use a serial implementation of all methods; thus, we do not leverage parallelism of the integrator. The reference solution is obtained by integrating the full
matrix differential equation with a high-accuracy adaptive solver, and errors
are measured at the final time $T$ as
\begin{align}\label{eq:final_time_error}
    \frac{\|\bfY(T)-\bfA_{\rm ref}(T)\|}{\|\bfA_{\rm ref}(T)\|}\,.
\end{align}

\paragraph{Interpretation of the results.}
 {The resulting convergence behaviour is recorded in Figure~\ref{fig:num_nonstiff_schroedinger}, where results in the left column are computed with an underlying Runge--Kutta method of fixed order (RK-$8$). Results in the right column are computed with a Runge--Kutta method of matching order. We note that the RK-BUG method does not require solving partial differential equations and thus does not allow the freedom to choose a specific time integration method for the substeps. It is observed that the use of RK-$8$ can have two effects. First, the error constant is significantly decreased, especially for larger ranks. Second, the convergence rate can be increased when the underlying time integration method of the resulting matrix ODEs dominates convergence. An increased order of convergence can, for example, be seen for the BUG-$3$ method with RK-$8$, where the convergence matches that of the BUG-$4$ method. These two observations underline, already for non-stiff problems, the benefit of decoupling the time integration method from the underlying integrator used to construct the basis. The corresponding efficiency is depicted in Figure~\ref{fig:efficiency_nonstiff_schroedinger}. We wish to underline that the runtime of each method is highly implementation- and hardware-dependent. In our numerical experiments, we observe that the BUG-$p$ integrator tends to outperform both the RK-BUG and the parallel BUG-$p$ methods. This is expected, since the parallel BUG-$p$ method does not perform truncations in the basis construction, which is expected to be beneficial for GPU parallelization, while the differential equations are evolved sequentially. Moreover, we anticipate this integrator to be especially relevant for tensor-valued differential equations or when combined with multistep methods. The RK-BUG integrator exhibits larger error constants (compared to the BUG-$p$ method used with an RK-$8$ time integrator) despite using additional basis matrices. The effect of the smaller number of basis functions becomes evident when comparing the BUG-$4$ integrator with RK-$4$ against the RK-BUG method. This comparison reveals costs from additional bases, since RK-BUG uses the same inherent time integrator (RK-$4$).}

\subsection{Stiff heat equation}
\label{sec:num_stiff_heat}

In the following, we highlight the methods' behaviours for a stiff test case. While our assumptions do not hold in such a setting, it has been observed in a large number of works that BUG integrators show good convergence behaviour in these settings. A rigorous error bound is, however, currently unavailable for any kind of BUG integrator. Following \cite{ceruti2024robust}, we consider a forced heat equation on the
two-dimensional torus $\torus^2=[-\pi,\pi)^2$. Let
$a:[0,T]\times\torus^2\to\R$, $T=2$, denote the solution. The continuous problem
is
\begin{equation}
\label{eq:num_stiff_heat_pde}
    \partial_t a(t,x,y)
    =
    \frac12\partial_{xx}a(t,x,y)
    +
    \frac12\partial_{yy}a(t,x,y)
    +
    g(x,y),
    \qquad (x,y)\in\torus^2,
\end{equation}
with periodic boundary conditions in both spatial variables. The forcing is
time independent and is given by
\begin{equation}
\label{eq:num_stiff_heat_source}
    g(x,y)
    =
    \sum_{k=1}^{11}
    10^{-(k-1)}
    \exp\!\left(-k(x^2+y^2)\right).
\end{equation}
The initial condition is
\begin{equation}
\label{eq:num_stiff_heat_initial}
    a(0,x,y)=\sin(x)\sin(y).
\end{equation}

The spatial discretization is a standard tensor-product centered finite
difference discretization with periodic boundary conditions on an $N\times N$ grid,
with $N=128$, where, in contrast to the non-stiff results in Section~\ref{sec:num_stiff_heat}, the discrete Laplacian is scaled with the factor
$\Delta x^{-2}$. This yields a stiff matrix differential equation of the form
\eqref{eq:origProb},
\[
    \dt{\bfA}(t)=\bfF(\bfA(t)),
    \qquad
    \bfA(t)\in\R^{N\times N},
\]
where $\bfA(t)$ contains the nodal values of $a(t,\cdot,\cdot)$. In compact
matrix notation, the semi-discrete equation is
\begin{equation}
\label{eq:num_stiff_heat_matrix_ode}
    \dt{\bfA}(t)
    =
    -\bigl(\bfD\bfA(t)+\bfA(t)\bfD^\top\bigr)
    +
    \bfG,
\end{equation}
where $\bfD$ denotes the one-dimensional periodic second-difference matrix
representing $-\frac12\partial_{xx}$ and $\bfG$ contains the sampled values of
the source term \eqref{eq:num_stiff_heat_source}.

The low-rank initial value is written in the form
$
    \bfY_0=\bfU_0\bfS_0\bfV_0^\top,
$
where the factors are constructed from sine modes and represent the nodal values
of \eqref{eq:num_stiff_heat_initial}. Although the implementation embeds this
initial value in a larger low-rank basis, the nonzero part corresponds to the
rank-one function $\sin(x)\sin(y)$.

For the time integration, we focus only on the proposed BUG-$p$ integrator as presented in Section~\ref{sec:BUG-p-intro}. In contrast to the
non-stiff Schr\"odinger experiment, the reduced evolution equations are solved
with stiffly stable exponential time integrators for the coefficient update. A comparison to the explicit Runge--Kutta BUG method of
\cite{nobile2025robust} is not included in this stiff test, since its explicit
time stepping is subject to the parabolic stability restriction induced by
\eqref{eq:num_stiff_heat_matrix_ode}. The reference solution is obtained by
integrating the full matrix equation with an order $8$ full-order
solver, and errors are measured at the final time $T$ using \eqref{eq:final_time_error}. 

 {To showcase the flexibility of the BUG-$p$ integrator, we construct the basis according to Ketcherson's low-storage method of order $6$, which in projected form reads
\begin{align*}
\bfY^{(0)} &= \bfY^n, \\[4pt]
\bfY^{(1)} &= [\![\bfY^{(0)} + \tfrac16 h\, \bfP(\bfY^{(0)})\bfF(\bfY^{(0)})]\!], \\
\bfY^{(i+1)} &= [\![\bfY^{(i)} + \tfrac16 h\, \bfP(\bfY^{(i)})\bfF(\bfY^{(i)})]\!], \qquad i = 1,2,3, \\[4pt]
\bfY^{(5)} &= [\![\tfrac35 \bfY^{(0)} + \tfrac25 \Big(\bfY^{(4)} + \tfrac16 h\, \bfP(\bfY^{(4)})\bfF(\bfY^{(4)})\Big)]\!], \\
\bfY^{(i+1)} &= [\![\bfY^{(i)} + \tfrac16 h\, \bfP(\bfY^{(i)})\bfF(\bfY^{(i)})]\!], \qquad i = 5,6,7,8, \\[6pt]
\bfY^{n+1} &= \frac1{25}\bfY^{n}+\frac9{25}\Big(\bfY^{(4)} + \tfrac16 h\, \bfP(\bfY^{(4)})\bfF(\bfY^{(4)})\Big)+\frac35\Big(\bfY^{(9)} + \tfrac16 h\, \bfP(\bfY^{(9)})\bfF(\bfY^{(9)})\Big).
\end{align*}
Following the discussion in Section~\ref{sec:BUG-p-intro}, the basis is thus given as
 \begin{align*}
    \wt\bfU =\,& \text{orth}([\bfU_n, \bfU^{(4)}, \bfF(\bfY^{(4)})\bfV^{(4)}, \bfU^{(9)}, \bfF(\bfY^{(9)})\bfV^{(9)}\}\,,\\
    \wt\bfV =\,& \text{orth}([\bfV_n, \bfV^{(4)}, \bfF(\bfY^{(4)})^{\top}\bfU^{(4)}, \bfU^{(9)}, \bfF(\bfY^{(9)})^{\top}\bfU^{(9)}\}\,.
\end{align*}  
This gives an order $6$ method with only a rank $5r$ basis. While the basis construction requires the computation of several substeps and singular value decompositions, the costs of the $S$-step are thus minimized. This becomes relevant when the $S$-step computation is expensive, e.g., due to the use of implicit or exponential methods. Here, we employ a sixth-order exponential method to evolve the $\bfS$ matrix. Relative errors can be found in Figure~\ref{fig:num_nonstiff_schroedinger_bugp} together with the corresponding efficiency in Figure~\ref{fig:num_nonstiff_schroedinger_parallel_bugp}.  While the sixth-order BUG integrator does not exhibit clean order six convergence, we underline that the used exponential integrator also does not lead to a clean convergence order for the full rank problem. Moreover, we wish to again underline that this test case is not covered by our numerical analysis, since it does not fulfill our main assumptions used in the numerical analysis.
}

\begin{figure}[t]
    \centering

    \begin{subfigure}[t]{0.49\textwidth}
        \centering
        \includegraphics[width=\textwidth]{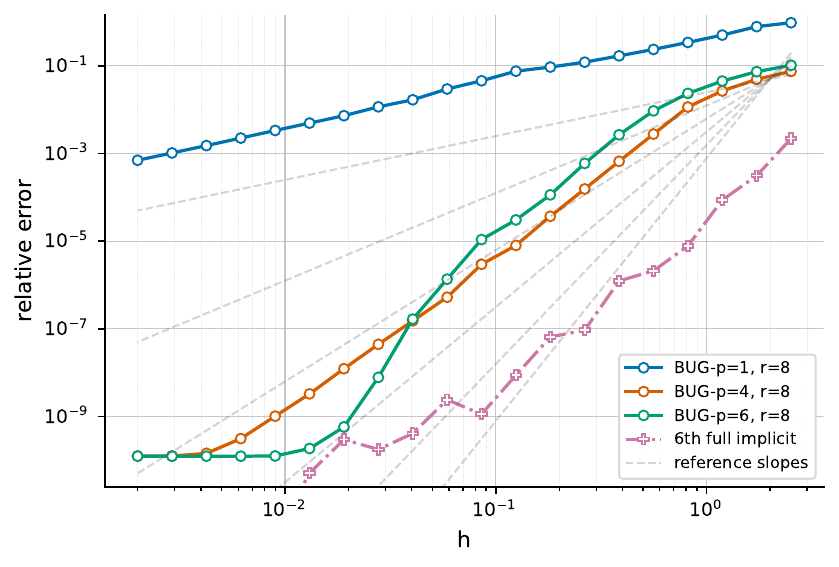}
        \caption{Convergence for the stiff heat equation
    \eqref{eq:num_stiff_heat_pde}. The panels show the
    final-time relative error for the proposed BUG-$p$
    method compared to the full-rank system integration.}
        \label{fig:num_nonstiff_schroedinger_bugp}
    \end{subfigure}
    \hfill
    \begin{subfigure}[t]{0.49\textwidth}
        \centering
        \includegraphics[width=\textwidth]{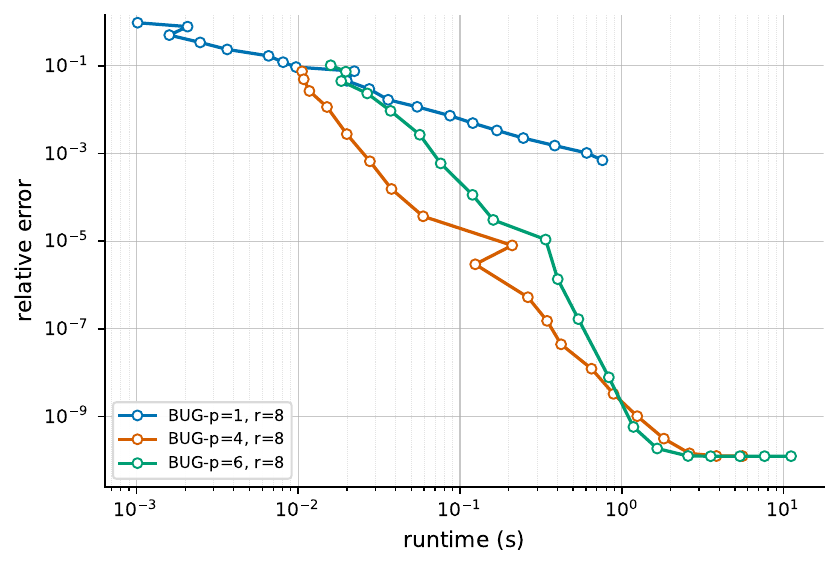}
        \caption{Efficiency of the BUG-p integrators of different order. Each marker denotes one time-step choice.  BUG-p=4 has the best error-to-runtime inflection point.}
        \label{fig:num_nonstiff_schroedinger_parallel_bugp}
    \end{subfigure}
    \label{fig:num_stiff_heat}
\end{figure}

\section{Summary and conclusion}\label{sec:summary_conculsion}

 {We have introduced a framework for constructing basis-update \& Galerkin (BUG) integrators of arbitrary order $p$ for dynamical low-rank approximation. Unlike previous general-order approaches, our integrators do not tie the coefficient update to the explicit Runge-Kutta scheme used for basis construction: the auxiliary RK method serves only to generate high-order basis directions, while the reduced $S$-, $K$-, and $L$-step equations may be integrated with any scheme of matching or higher order, including implicit or exponential methods. This is, to our knowledge, the first general-order BUG construction compatible with stiffly stable time integration.

We derived two variants of this construction -- an augmented BUG-p integrator and a parallel BUG-p integrator -- and proved a robust, curvature-independent error bound of order p for both (Theorem~\ref{th:robusterrorp}), under standard Lipschitz, boundedness, and small-curvature assumptions. A key technical ingredient (Lemma~\ref{le:approxU}) shows that the RK stage information needed for a $p$-th order basis is already contained in a basis of rank $(s+1)r$ for the augmented method and $sr$ for the parallel method, which is smaller by roughly a factor of two than the $2rs^*$ basis required by the explicit RK-BUG integrator of \cite{nobile2025robust}, at comparable or better accuracy.

Numerically, we confirmed the predicted convergence rates for both integrators on a non-stiff Schrödinger benchmark, and showed that decoupling the coefficient integrator from the basis-generating RK method can be exploited in two distinct ways: to reduce the error constant, and, when the coefficient integrator has higher order than the basis-generating RK method, to increase the observed convergence rate beyond the integrator's nominal order. We further demonstrated on a stiff forced heat equation that an order-six augmented BUG integrator, using a six-stage low-storage RK basis construction paired with an exponential S-step integrator, remains stable and convergent at large time steps where the underlying assumptions of Theorem 1 are not formally satisfied. Since a rigorous error bound for the stiff regime remains open, our numerical evidence should be read as an empirical, not proven, result.

The formulation of general conditions on the basis construction opens the door to the construction of further BUG-$p$ integrators, e.g., through multistep methods which could especially become relevant for the parallel BUG-$p$ integrator.}

\section{Acknowledgements}
This material is based upon work supported by the U.S. Department of Energy, Office of Science, Office of Advanced Scientific Computing Research, as part of their Applied Mathematics Research Program. The work was performed at the Oak Ridge National Laboratory, which is managed by UT-Battelle, LLC under Contract No. De-AC05-00OR22725. The United States Government retains and the publisher, by accepting the article for publication, acknowledges that the United States Government retains a non-exclusive, paid-up, irrevocable, world-wide license to publish or reproduce the published form of this manuscript, or allow others to do so, for the United States Government purposes. The Department of Energy will provide public access to these results of federally sponsored research in accordance with the DOE Public Access Plan (http://energy.gov/downloads/doe-public-access-plan).


	\bibliographystyle{abbrv}
	\bibliography{main} 

\appendix
\newpage
\section{Notation}\label{tab:notation}

\begin{longtable}{p{0.33\linewidth} p{0.63\linewidth}}
\caption{Core notation used throughout the manuscript.}
\label{tab:notation}
\\
\hline
Symbol & Meaning \\
\hline
\endfirsthead

\multicolumn{2}{l}{\small\textit{Table \thetable{} continued from previous page}} \\
\hline
Symbol & Meaning \\
\hline
\endhead

\hline
\multicolumn{2}{r}{\small\textit{Continued on next page}} \\
\endfoot

\hline
\endlastfoot

\multicolumn{2}{l}{\textit{Continuous problem and low-rank manifold}}\\
\hline
$\bfA(t)\in\R^{m\times n}$ 
& Exact/full-rank solution of $\dt{\bfA}(t)=\bfF(t,\bfA(t))$. \\

$\bfF(t,\bfZ)$
& Matrix-valued right-hand side evaluated at time $t$ and state $\bfZ$. \\

$\bfX(t)=\bfU(t)\bfS(t)\bfV(t)^\top$
& Generic rank-$r$ dynamical low-rank approximation. \\

$\mathcal{M}_r$
& Manifold of matrices of rank $r$. \\

$\bfP(\bfX)$
& Orthogonal projector onto the tangent space of $\mathcal{M}_r$ at $\bfX=\bfU\bfS\bfV^\top$. \\

$(\bfI-\bfP(\bfY))\bfF(t,\bfY)$
& Normal component of the vector field at $\bfY$; assumed to be $O(\eps_r)$. \\

$L,\;B,\;\eps_r,\;\delta$
& Lipschitz constant, uniform bound for $\bfF$, bounds on the normal component, and initial error bound. \\

$p$
& Target order of convergence of the BUG-$p$ integrator. \\

$s,\;s^*$
& Number of stages of the underlying explicit Runge--Kutta method; number of stages with $b_i\neq0$. \\

$\mu = h(h^p+\eps_r)$
& Combined local error scale used throughout the error analysis. \\

$\bfR(t) = \bfA(t)-\bfU\bfU^\top\bfA(t)\bfV\bfV^\top$
& Residual (normal-space) part of the exact solution with respect to a given basis pair $(\bfU,\bfV)$. \\

$\bfD(t)$
& Defect term $\bfF(t,\bfU\widetilde\bfS(t)\bfV^\top+\bfR(t))-\bfF(t,\bfU\widetilde\bfS(t)\bfV^\top)$ used in the Gronwall argument. \\

\hline
\multicolumn{2}{l}{\textit{Discrete low-rank solution and truncation}}\\
\hline
$t_0,\;t_1=t_0+h$
& Initial time, next time level. \\

$h$
& Time step size. \\

$\bfY_n=\bfU_n\bfS_n\bfV_n^\top$
& Numerical low-rank approximation at time $t_n$; $\bfU_n$ and $\bfV_n$ have orthonormal columns. \\

$r$
& Current rank. \\

$[\![\bfZ]\!], [\![\bfZ]\!]_{r_n}, [\![\bfZ]\!]_{\vartheta}$
& Truncation of $\bfZ$ to rank $r$, to a specific $r_n$, with truncation tolerance $\vartheta$. \\

$\vartheta$
& Truncation error introduced in the final compression step. \\

${\rm orth}(\cdot)$
& Orthonormal basis of the column span of the supplied matrix block. \\

\hline
\multicolumn{2}{l}{\textit{Augmented and parallel BUG integrators}}\\
\hline
$\overline{\bfU},\;\overline{\bfV}$
& Augmented bases used in the augmented BUG integrator. \\

$\overline{\bfS}(t)$
& Reduced coefficient matrix in the augmented BUG Galerkin step and the Parallel BUG coefficient step. \\

$\wh\bfU_0,\;\wh\bfV_0$
& Pre-augmented bases in the parallel BUG integrator. \\

$\bfK(t),\;\bfL(t)$
& Solution matrices of the parallel $K$- and $L$-steps. \\

$\wh\bfU=[\wh\bfU_0,\widetilde{\bfU}_2],\;
 \wh\bfV=[\wh\bfV_0,\widetilde{\bfV}_2]$
& Final augmented bases after the parallel basis-update step. \\

$\bfM=\overline\bfU^\top\bfU_n,\;\bfN=\overline\bfV^\top\bfV_n$
& Projection matrices mapping the old basis into the augmented basis (augmented BUG integrator). \\

$\wh \bfM=\wh\bfU_0^\top\bfU_n,\;\wh \bfN=\wh\bfV_0^\top\bfV_n$
& Projection matrices mapping the old basis into the pre-augmented basis (parallel BUG integrator). \\

$\bfS_{n+1}^{\rm par}$
& Assembled coefficient matrix of the parallel integrator in the final augmented basis $(\wh\bfU,\wh\bfV)$. \\

\hline
\multicolumn{2}{l}{\textit{Runge--Kutta stage notation for the order-$p$ construction}}\\
\hline
$a_{\ell,\ell'},\;b_\ell,\;c_\ell$
& Butcher coefficients of the underlying explicit $s$-stage Runge--Kutta method. \\

$t_{(\ell)}=t_k+c_\ell h$
& Time associated with stage $\ell$. \\

$\bfY_{(\ell)}=\bfU_{(\ell)}\bfS_{(\ell)}\bfV_{(\ell)}^\top$
& Projected low-rank stage value. \\

$\bfF_{(\ell)}=\bfF(t_{(\ell)},\bfY_{(\ell)})$
& Right-hand side evaluated at stage $\ell$. \\

$\bar{\bfZ}_\ell=\bfP(\bfY_{(\ell)})\bfF_{(\ell)}$
& Projected stage increment (tangent vector). \\

$\widetilde \bfZ_{\ell}, \widetilde \bfZ_{\ell}^{(S)}$
& Stage increment of the $K$-step and $S$-step. \\

$\bfY_1^{(K)}, \bfY_1^{(S)}$
& Updated $K$ and $S$-step solutions using an order $p$ RK integrator. \\

$\wh\bfU_{(j)},\;\wh\bfV_{(j)}$
& Recursively augmented bases collecting stage information up to stage $j$. \\

$\bfZ_j,\;\bfA_1$
& Full-rank (unprojected) $j$-th Runge--Kutta stage increment and endpoint value, used only in the error analysis. \\

\end{longtable}

\section{Recap: High-order BUG integrators}\label{app:recap}

Second--order BUG integrators have recently been derived in \cite{ceruti2024robust} {as an extension to the augmented BUG integrator and in \cite{kusch2024second} as an extension to the parallel BUG integrator}. The core idea of these integrators is to use the augmentation steps developed in~\cite{CeKL22} to augment the basis with higher-order information. In the following, we restate the parallel and augmented BUG integrators. Our presentation focuses on a time update from an initial time $t_0$ to the next time step $t_1 = t_0 + h$, where $h$ denotes the user-determined time step size. We denote the factorized numerical solution at time $t_n$ as $\bfY_n = \bfU_n\bfS_n\bfV_n^{\top}$ where $n$ can take real values, e.g., $n=1/2$ to denote the numerical solution at the half-step $t_{1/2} = t_0 + h/2$.

\subsection{Midpoint BUG integrator}
The midpoint BUG integrator \cite{ceruti2024robust} first adds midpoint information, followed by a Galerkin update. More specifically, the integrator computes a step with step size $h/2$ with the (first-order) augmented BUG integrator (sometimes also referred to as the rank-adaptive BUG integrator) of \cite{ceruti2024robust} without the truncation step to compute an approximation of the temporal midpoint $t_{1/2}$ with rank $r \le\wh r \le 2r$,
$$
\wh \bfY_{1/2}=\wh\bfU_\h \wh\bfS_\h \wh\bfV_\h^\top \approx \bfA(t_\h).
$$ 
If $\bfY_{1/2}$ is truncated to rank $r$, this yields the rank $3r$ variant of the midpoint BUG integrator. Without truncation, the resulting integrator will have rank $4r$ basis matrices. Given this midpoint approximation, an augmented basis is constructed as
\begin{equation}\label{eq:Ubar-Vbar}
\begin{aligned}
\wt \bfU &= {\rm orth}(\wh\bfU_\h, h\bfF(t_\h,\wh \bfY_\h) \wh\bfV_\h) \ \hbox{ and }\\
\wt \bfV &= {\rm orth}(\wh\bfV_\h, h\bfF(t_\h,\wh\bfY_\h)^\top  \wh\bfU_\h)\,,
\end{aligned}
\end{equation}
where ${\rm orth}$ denotes an orthonormalization with, e.g., a Gram-Schmidt algorithm. The integrator then takes the following form (for the rank $4r$ variant):

\begin{enumerate}
\item {\bf Basis augmentation:} Compute the augmented orthonormal bases $\wt \bfU$ and $\wt \bfV$ of rank $\wt r \le 4r$ according to \eqref{eq:Ubar-Vbar} and the $\wt r \times r$ matrices $\wt \bfM=\wt \bfU^\top \bfU_0$ and $\wt\bfN=\wt \bfV^\top \bfV_0$.
\item {\bf Galerkin step:} Integrate, from $t=t_0$ to~$t_1$, the $\wt r \times \wt r$ matrix differential equation
		\begin{equation}\label{bar-S-step} 
		\dt{\wt\bfS}(t) =  \wt\bfU^\top \bfF(t, \wt\bfU \, \wt\bfS(t) \wt\bfV^\top) \wt\bfV, 
		\qquad \wt\bfS(t_0) = \wt\bfM \bfS_0 \wt\bfN^\top.
		\end{equation}
\item \textbf{Truncate}: Truncate the factorized solution to a new rank $r_1$, see \cite{ceruti2023rank}.
\end{enumerate}

\subsection{Second--order parallel integrator}
The second--order parallel integrator is constructed as an extension to the first--order variant \cite{CeKL23}, where instead of the basis at the old time step, the integrator uses the pre-augmented bases
\begin{align}\label{eq:pre-aug-bases}
    \wh\bfU_0 = \text{orth}([\bfU_0, \bfF(t_0,\bfY_0)\bfV_0]) \enskip\text{ and }\enskip \wh\bfV_0 = \text{orth}([\bfV_0, \bfF(t_0,\bfY_0)^{\top}\bfU_0])\,.
\end{align}
The integrator then takes the following form:

\begin{enumerate}
\item {\bf Pre-augmentation:} Construct the pre-augmented basis matrices $\wh\bfU_0\in \mathbb{R}^{m\times 2r}$ and $\wh\bfV_0\in \mathbb{R}^{n\times 2r}$ according to \eqref{eq:pre-aug-bases}.
\item {\bf Parallel update:} Determine the augmented basis matrices $\widehat{\mathbf{U}}\in \mathbb{R}^{m\times 4r}$ and $\widehat{\mathbf{V}}\in \mathbb{R}^{n\times 4r}$ as well as the coefficient matrix $\bar{\mathbf{S}}(t_1) \in\mathbb{R}^{2r \times 2r}$ (in parallel):
\\[2mm]
\textbf{K-step}: Integrate from $t=t_0$ to $t_1$ the $m \times 2r$ matrix differential equation
\begin{align*}
\dt{\bfK}(t) =\,& \bfF(t,\bfK(t)\widehat \bfV_0^{\top})\widehat \bfV_0\, , \quad \bfK(t_0) = \bfU_0 \bfS_0 \bfV_0^{\top}\widehat{\bfV}_0\, .
\end{align*}
Construct $\widehat \bfU= [\widehat \bfU_0, \widetilde \bfU_2] = \text{orth}([\widehat \bfU_0, \bfK(t_1)])$. 
\\[2mm]
\textbf{L-step}: Integrate from $t=t_0$ to $t_1$ the $n \times 2r$ matrix differential equation
\begin{align*}
\dt{\bfL}(t) =\,& \bfF(t,\widehat \bfU_0 \bfL(t)^{\top})^{\top}\widehat \bfU_0\, , \quad \bfL(t_0) = \bfV_0 \bfS_0^{\top} \bfU_0^{\top}\widehat{\bfU}_0\, .
\end{align*}
Construct $\widehat \bfV= [\widehat \bfV_0, \widetilde \bfV_2] = \text{orth}([\widehat \bfV_0, \bfL(t_1)])$. 
\\[2mm]
\textbf{S-step}: Solve the $2r \times 2r$ matrix differential equation
\begin{align*}
\dt{\bar \bfS}(t) = \widehat \bfU_0^{\top}\bfF(t,\widehat{\bfU}_0\bar \bfS(t)\widehat{\bfV}_0^{\top})\widehat \bfV_0\, , \quad \bar \bfS(t_0) = \widehat{\bfU}_0^{\top}\bfU_0 \bfS_0 \bfV_0^{\top}\widehat{\bfV}_0\, .
\end{align*}
\item \textbf{Augment}: Set up the augmented coefficient matrix $\widehat{\bfS}_1 \in\mathbb{R}^{4r\times 4r}$ as
\begin{align*}
\widehat{\bfS}_1 = \begin{pmatrix}
\bar{\bfS}(t_1) & \bfL(t_1)^{\top}\widetilde \bfV_2\\
\widetilde \bfU_2^{\top} \bfK(t_1) & \bf0
\end{pmatrix}\, .
\end{align*}
\item \textbf{Truncate}: Truncate the factorized solution to a new rank $r_1$, see \cite{ceruti2023rank}.
\end{enumerate}

\end{document}